\documentclass[a4paper]{amsart}
\usepackage{amsmath}
\usepackage{amsthm}
\usepackage{amssymb}
\usepackage[all]{xypic}
\usepackage{color}
\usepackage{hyperref}
\usepackage{tabularx}

\newtheorem{theo}{Theorem}
\newtheorem{ex}[theo]{Example}
\newtheorem{lem}[theo]{Lemma}
\newtheorem{defn}[theo]{Definition}

\newtheorem{cor}[theo]{Corollary}
\newtheorem{rem}[theo]{Remark}
\newtheorem{prop}[theo]{Proposition}

\newcommand{\HP}{\operatorname{HP}}
\newcommand{\HH}{\operatorname{HH}}
\newcommand{\ext}{\operatorname{Ext}}

\newcommand{\Hom}{\operatorname{Hom}}
\newcommand{\Tor}{\operatorname{Tor}}
\newcommand{\im}{\operatorname{im}}

\newcommand{\ann}{\operatorname{Ann}}
\newcommand{\rank}{\operatorname{rank}}
\newcommand{\Alt}{\operatorname{Alt}}
\newcommand{\opp}{\operatorname{op}}

\newcommand{\sx}{\mathsf{x}}
\newcommand{\sy}{\mathsf{y}}
\newcommand{\sz}{\mathsf{z}}
\newcommand{\sA}{\mathsf{A}}
\newcommand{\sB}{\mathsf{B}}
\newcommand{\sd}{\mathsf{d}}
\newcommand{\sa}{\mathsf{a}}
\newcommand{\bb}{\mathsf{b}}
\newcommand{\scc}{\mathsf{c}}
\newcommand{\sff}{\mathsf{f}}

\newcommand{\sM}{\mathsf{M}}
\newcommand{\sR}{\mathsf{R}}
\newcommand{\cA}{\mathcal{A}}
\newcommand{\cB}{\mathcal{B}}
\newcommand{\cC}{\mathcal{C}}
\newcommand{\kq}{{k[q^{\pm 1}]}}

\author{Matthew Towers}
\title{Poisson and Hochschild cohomology and the semiclassical limit}
\email{m.towers@kent.ac.uk}
\address{School of Mathematics, Statistics and Actuarial Science,
University of Kent, UK.}
\begin{document}
\maketitle
%\definecolor{grey}{gray}{.75} \pagestyle{myheadings}
%\settimeformat{xxivtime} \markboth{ {\textcolor{grey}{ Version of
%      \today \, at \currenttime}}}{ \textcolor{grey}{ Version of
%    \today \, at \currenttime}}
\begin{abstract}
  Let $\sA$ be a quantum algebra possessing a semiclassical limit $A$.
  We show that under certain hypotheses $\sA^e$ can be thought of as a deformation of the Poisson enveloping algebra of $A$, and we give a
  criterion for the Hochschild cohomology of $\sA$ to be a deformation
  of the Poisson cohomology of $A$ in the case that $\sA$ is Koszul. We
  verify that condition for the algebra of $2\times 2$ quantum
  matrices and calculate its Hochschild
  cohomology and the Poisson cohomology of its semiclassical limit.
\end{abstract}

Certain quantum algebras $\sA$ admit a semiclassical limit: that is, a
Poisson algebra structure on their $q \to 1$ limit $A$ which preserves
some noncommutative information from $\sA$. A simple example is the
quantum plane
\[\mathsf{A}= \frac{k(q)\langle
  \mathsf{x},\mathsf{y}\rangle}{(\mathsf{xy}-q\mathsf{yx})}\] 
whose semiclassical limit is the polynomial ring
$k[x,y]$ with Poisson structure determined by 
$ \{x,y\} = xy$. Various authors have investigated 
to what extent the Hochschild (co)homology of $\sA$ is determined by
the Poisson (co)homology of its semiclassical limit, for example  \cite{vdb,FT},
usually using spectral sequence methods, and a related situation where
there is a Poisson structure on an associated graded algebra is 
dealt with in \cite[Th\'eor\`eme, p.223]{Kassel}.  The famous Kontsevich
quantization theorem also guarantees a relationship between the
cohomology of a Poisson algebra and that of its canonical
quantization.

In this paper we study the link between Poisson and Hochschild
cohomology by showing that if $\sA$ is graded with a PBW basis of
polynomial type with a polynomial semiclassical limit $A$, the
enveloping algebra $\sA^e$ is a deformation of the Poisson enveloping
algebra $P(A)$ (Section
\ref{scl}).  The restriction to $A$ being polynomial allows us 
to describe $P(A)$ by generators and relations
(Lemma \ref{pea}): when the K\"ahler differentials of $A$ are not free
as an $A$-module, additional relations are required.

If $\sA$ is Koszul with quadratic dual $\sA^!$ the Hochschild
cohomology of $\sA$ is computed by a differential graded algebra whose
underlying algebra is $\sA\otimes \sA^!$. 
If $\sA$ has a PBW basis of polynomial type we show that this DGA is a
deformation of the DGA $A\otimes A^!$ that computes Poisson cohomology
of the semiclassical limit $A$ (Proposition \ref{defm_of_Kos_cocomp}).
This leads to a condition for  $\HH(\sA)$ to be a deformation of
$\HP(A)$; in particular when this holds the Hochschild and Poisson
cohomologies have the same  bigraded Hilbert series (Corollary
\ref{HH_eq_HP}).

Two algebras to which our results can be applied are the coordinate ring
$\sA(n)$ of quantum affine $n$-space and the algebra $\sM$ of
$2\times 2$ quantum matrices.  
We discuss the cohomology of the quantum plane and its semiclassical
limit in Section
\ref{sec:quantum_plane} as an example of our methods; Hochschild
cohomology for quantum affine spaces and Poisson cohomology of their
semiclassical limits are already known, for example \cite[\S6]{Wambst},
\cite{BGM}, \cite[Proposition 2.2.1]{LR}, \cite[\S3.3]{Sitarz}.

In Section \ref{sec:matrices} we show that the Hochschild cohomology for
$2\times 2$ quantum matrices is a $q$-deformation of
the Poisson cohomology of the semiclassical limit, and calculate this
Poisson cohomology explicitly in Theorem \ref{HPM}.  
These computations are new, although some low-dimensional Hochschild cohomology groups for quantum matrices
have appeared in the literature:  the zeroth
Hochschild cohomology group is known to be generated by the quantum determinant, 
and the structure of the first cohomology group as a module over the
centre was determined in \cite{LL}.
An interesting feature of these computations is the action of the
Poisson centre
on the Poisson cohomology groups, which acts freely except for a single
trivial summand in the top cohomological dimension.  This is in contrast
to \cite[Theorem 4.1]{vdb} for example, where the action is free.

\subsection{Notation and conventions} Throughout this article $q$ is
a transcendental
element over a field $k$ of characteristic zero. $k[q^{\pm 1}]$ is the
ring of Laurent polynomials in $q$ and $k(q)$ its field of
fractions.  ``Graded''  means $\mathbb{Z}_{\geq 0}$-graded unless
otherwise stated. The $n$th graded component of a graded algebra 
$\Lambda$ is denoted $\Lambda_n$, and if $\lambda \in \Lambda_n$ we 
write $|\lambda|=n$. 

\section{Poisson algebras and Poisson cohomology} \label{env_alg}

\subsection{The Poisson enveloping algebra}

Let $A$ be a Poisson $k$-algebra with bracket $\{-,-\}$; that is, $A$
is a commutative associative $k$-algebra with unit, $\{-,-\}$ is a Lie
bracket on $A$, and for any $a \in A$ the map \[ b \mapsto \{a,b\}\]
is a derivation.  A \textbf{left Poisson module} over $A$ is a
$k$-vector space $M$ which is a simultaneously a left module for $A$
as an associative algebra and a left module for the Lie algebra $(A,
\{-.-\})$, satisfying
\begin{align*}
\{x,y  m\}_M &= \{x,y\} m + y  \{x,m\}_M \\
\{xy,m\}_M &= x\{y,m\}_M + y \{x,m\}_M
\end{align*}
for any $x,y \in A$ and $m \in M$, where $\{a,m\}_M$ denotes the Lie
algebra action and $am$ the associative algebra action of $a \in A$ on
$m \in M$.

Let $\Omega(A)$ be the $A$-module of \textbf{K\"ahler differentials}
of $A$.  This is the free left $A$-module on generators $\Omega(a)$
for $a \in A$ quotiented by the submodule generated by all elements of
the form
\begin{equation} \label{k_diffs} \Omega(\alpha), \,\,\,\,\,\,
 \Omega(a+b)-\Omega(a)-\Omega(b), \,\,\,\,\,\,
 \Omega(ab)-a\Omega(b)-b\Omega(a) 
\end{equation}
for $\alpha\in k$ and $a,b \in A$.  If we define
\[ [ a\Omega(x),b\Omega(y) ] = a\{x,b\}\Omega(y) + b\{a,y\}\Omega(x) +
ab\Omega(\{x,y\}) \] then $\Omega(A)$ becomes a $k$-Lie algebra:
see \cite[Theorem 3.8]{Hueb}.

There is an associative algebra $U(A,\Omega(A))$ called the
\textbf{Poisson enveloping algebra} satisfying a universal property
such that the category of left $U(A,\Omega(A))$-modules is equivalent
to the category of left Poisson modules over $A$.  A more general
construction of which the Poisson enveloping algebra is a special case
given in \cite[\S1]{Hueb}.

In the rest of this section $A=k[x_1,\ldots,x_n]$ will be a polynomial
algebra so that the K\"ahler differentials are
freely generated as an $A$-module by $\Omega(x_1),\ldots,\Omega(x_n)$
\cite[\S8.8]{Weibel}. 
We will need a presentation of the Poisson enveloping algebra by
generators and relations in this special case.  To this end, let
$P(A)$ be the  $k$-algebra generated by $y_i$ and $\Omega(y_i)$ for $1 \leq i
\leq n$, subject to two sets of relations.  The first set says that
the $y_i$s commute:
\begin{equation}\label{rel1}
y_i y_j = y_j y_i \,\,\, 1 \leq i < j \leq n
\end{equation}
so there is an algebra homomorphism $\iota_A: A \to P(A)$ determined
by $x_i \mapsto y_i$.  The second set of relations is
\begin{gather} 
\label{rel2} \Omega(y_i) y_j - y_j \Omega(y_i) = \iota_A(\{x_i,x_j\})
\,\,\,\,\,\, 1 \leq i,j \leq n \\
\label{rel3} \Omega(y_i)\Omega(y_j)-\Omega(y_j)\Omega(y_i)=
\sum_k \iota_A\left(\frac{ \partial \{x_i,x_j\}}{\partial
    x_k}\right)\Omega(y_k)  \,\,\,\,\,\, 1\leq i < j \leq n.
\end{gather}
The map $\iota_A: A \to P(A)$ makes $P(A)$ into an
$A$-module.  Let $\iota_{\Omega(A)} : \Omega(A) \to P(A)$ be the map of $A$-modules such
that  $\Omega(x_i) \mapsto \Omega(y_i)$, and regard $P(A)$ as a $k$-Lie
algebra by defining $[u,v] = uv-vu$.

\begin{lem} $\iota_{\Omega(A)}$ is a homomorphism of $k$-Lie
  algebras. \end{lem}
\begin{proof}
Firstly note that
\begin{equation}
\Omega(y_i) \iota_A(b) - \iota_A(b) \Omega(y_i) =
  \iota_A(\{x_i,b\})
  \label{identity}
\end{equation} for any $b \in A$.  When $b$ is a monomial this
can be proved by induction on the length using (\ref{rel2}), and the
result extends to arbitrary $b \in A$ by linearity.

To prove that $\iota_{\Omega(A)}$ is a Lie algebra homomorphism we
must show that for any $a,b \in A$ and $1 \leq i,j \leq n$,
\begin{multline*} \iota_{\Omega(A)} ([a\Omega(x_i), b\Omega(x_j)])=
  \iota_A(a\{x_i,b\})\Omega(y_j) + \iota_A(b\{a,x_j\})\Omega(y_i) \\+
  \iota_A(ab)\sum_k \iota_A\left(\frac{ \partial \{x_i,x_j\}}{\partial x_k}\right)\Omega(y_k)
\end{multline*}
is equal to 
\begin{multline*} \iota_{\Omega(A)}(a\Omega(x_i))
  \iota_{\Omega(A)}(b\Omega(x_j)) - \iota_{\Omega(A)}(b\Omega(x_j))
  \iota_{\Omega(A)}(a\Omega(x_i)) =\\
  \iota_A(a)\Omega(y_i)\iota_A(b)\Omega(y_j)-
  \iota_A(b)\Omega(y_j)\iota_A(a)\Omega(y_i) \end{multline*} 
for any $a,b \in A$ and any $i,j$.

Using (\ref{identity}) gives
\begin{multline*}
  \iota_A(a)\Omega(y_i)\iota_A(b)\Omega(y_j)-
  \iota_A(b)\Omega(y_j)\iota_A(a)\Omega(y_i)  \\
= \iota_A(a)(\iota_A(b)\Omega(y_i) + \iota_A(\{x_i,b\}))\Omega(y_j)-
\iota_A(b)(\iota_A(a)\Omega(y_j) + \iota_A(\{x_j,a\}))\Omega(y_i) \\
= \iota_A(ab)(\Omega(y_i)\Omega(y_j)-\Omega(y_j)\Omega(y_i))
+\iota_A(a\{x_i,b\})\Omega(y_j) + \iota_A(b\{a,x_j\})\Omega(y_i)
\end{multline*}
and the result follows from (\ref{rel3})
\end{proof}

\begin{lem} \label{pea} The Poisson enveloping
  algebra $U(A, \Omega(A))$ is isomorphic to $P(A)$.
\end{lem}

\begin{proof}
  We will verify that $(P(A), \iota_A, \iota_{\Omega(A)})$ has the
  universal property of \cite[1.6]{Hueb}.  We must show that given an
  associative 
  $k$-algebra $B$ equipped with the Lie bracket
  $[b_1,b_2]=b_1b_2-b_2b_1$, a morphism of $k$-Lie algebras
  $\phi_{\Omega(A)}: \Omega(A) \to B$ and a morphism of $k$-algebras
  $\phi_A : A \to B$ such that
\begin{flalign}
\label{univ1} \phi_A(a) \phi_{\Omega(A)}(b\Omega(x_i)) & = \phi_{\Omega(A)}(ab\Omega(x_i))
\\
\label{univ2} \phi_{\Omega(A)}(b\Omega(x_i))\phi_A(a) -
\phi_A(a)\phi_{\Omega(A)}(b\Omega(x_i)) &= \phi_A(b\{x_i,a\})
\end{flalign}
for any $a,b \in A$ and any $i$, there is a unique homomorphism $\Phi
: P(A) \to B$ of $k$-algebras such that $\Phi \circ \iota_A = \phi_A$
and $\Phi \circ \iota_{\Omega(A)} = \phi_{\Omega(A)}$.

If these are to hold, $\Phi$ must satisfy $\Phi(y_i) = \phi_A(x_i)$
and $\Phi(\Omega(y_i))= \phi_{\Omega(A)} (\Omega(x_i))$. Since the elements
$y_i$ and $\Omega(y_i)$ generate $P(A)$, if such a $\Phi$ exists it is
unique.

In order to show that such a $\Phi$ exists we need only show that the
elements $\phi_A(x_i)$ and $\phi_{\Omega(A)} (\Omega(x_i))$ satisfy
the relations
(\ref{rel1}), (\ref{rel2}) and (\ref{rel3}) of $P(A)$.   Certainly the
$\phi_A(x_i)$ commute since $\phi_A : A \to B$ is an algebra
homomorphism, so (\ref{rel1}) is satisfied.  Next,
\[ \phi_{\Omega(A)} (\Omega(x_i)) \phi_A(x_j) -\phi_A(x_j)\phi_{\Omega(A)}
(\Omega(x_i)) = \phi_A(\{ x_i, x_j\})
\]
by (\ref{univ2}),  % Since
and therefore (\ref{rel2}) is satisfied.  Finally, 
\begin{gather*} \phi_{\Omega(A)} (\Omega(x_i))\phi_{\Omega(A)} (\Omega(x_j))-
\phi_{\Omega(A)} (\Omega(x_j))\phi_{\Omega(A)} (\Omega(x_i)) \\=
\phi_{\Omega(A)} ([\Omega(x_i),\Omega(x_j)]) = \phi_{\Omega(A)}(\Omega(\{x_i,x_j\}) \end{gather*}
because $\phi_{\Omega(A)}$ is a homomorphism of Lie algebras, and
applying $\phi_{\Omega(A)}$ to 
\[ \Omega(\{x_i,x_j\}) = \sum_k \frac{ \partial \{x_i,x_j\}}{\partial x_k}\Omega(x_k)
 \]
and using (\ref{univ1}) gives
\[\phi_{\Omega(A)} (\Omega(\{x_i,x_j\})) = \sum_k
\phi_A\left( \frac{ \partial \{x_i,x_j\}}{\partial x_k}  \right) \phi_{\Omega(A)} (\Omega(x_k))\]
and so (\ref{rel3}) is also satisfied.
\end{proof}

\begin{cor} \label{poisson_pbw} $P(A)$ has a PBW basis consisting of all elements of the form 
\[ y_1^{a_1}\cdots y_n^{a_n} \Omega(y_1)^{b_1}\cdots
\Omega(y_n)^{b_n} \] for $a_i, b_i \geq 0$. \end{cor}

\begin{proof}
  This is a consequence of a result of Rinehart \cite[Theorem 1.9]{Hueb}
  which shows that
  the Poisson enveloping algebra is isomorphic as a vector space to
  the symmetric $A$-algebra on $\Omega(A)$ via the canonical map.  
\end{proof}

\subsection{Poisson cohomology}\label{subsec:poisson_cohomology}
In this section $A=k[x_1,\ldots,x_n]$ is once again a polynomial
algebra so that  $\Omega(A)$ is a free $A$-module, and we may consider its
$m$th $A$-exterior power $\Lambda_A^m (\Omega(A))$. This is the
quotient of the tensor product over $A$ of $m$ copies of $\Omega(A)$
by the submodule generated by all %pure
tensors with two equal
factors.  Write 
\[ \Omega(a_1) \wedge \cdots \wedge  \Omega(a_m) \]
for the image of the pure tensor
\[ \Omega(a_1) \otimes_A \cdots \otimes _A \Omega(a_m)\]
in $\bigwedge_A^m (\Omega(A))$.  As an $A$-module, $\bigwedge_A^m
(\Omega(A))$ is free on all elements of the form
\[ \Omega(x_{i_1}) \wedge \cdots \wedge  \Omega(x_{i_m}) \]
for $1 \leq i_1 < i_2 < \cdots < i_m \leq n$, so it is isomorphic to
$A \otimes _k \Lambda^m(V)$ where $\Lambda^m(V)$ is the $m$th exterior
power of the vector space $V$ spanned by the $\Omega(x_i)$.

\begin{defn} $\operatorname{Alt}_A^m(\Omega(A),A)$ is
$\Hom_A(\bigwedge_A^m (\Omega(A)), A)$ \end{defn}
This is what Huebschmann calls
the space of $A$-multilinear alternating functions from $\Omega(A)$ to $A$.
The direct sum $\operatorname{Alt}_A^*(\Omega(A),A)=\bigoplus_{m\geq 0}
\operatorname{Alt}_A^m(\Omega(A),A)$ is a differential graded algebra
when equipped with the  Cartan-Chevalley-Eilenberg
differential
\begin{multline*}   df(\Omega(x_{i_1}) \wedge \cdots \wedge
  \Omega(x_{i_m})) = \\  \sum_{j\geq 1}
(-1)^{j+1} \Omega(x_{i_j}) f( \Omega(x_{i_1}) \wedge \cdots \wedge
\widehat{\Omega(x_{i_j})} \wedge \cdots \wedge \Omega(x_{i_m})) \\ 
\mbox{} +
\sum_{1 \leq j<k \leq m }(-1)^{j+k} f([\Omega(x_{i_j}), \Omega(x_{i_k})]
\wedge \Omega(x_{i_1}) \wedge \cdots \wedge
\widehat{\Omega(x_{i_j})} \wedge \cdots \wedge
\widehat{\Omega(x_{i_k})} \wedge \cdots \wedge \Omega(x_{i_m})
)\end{multline*}
and the shuffle product
\begin{multline*} (f \wedge g) (\Omega(a_1)\wedge \cdots \wedge
  \Omega(a_{|f|+|g|})) = \\
  \sum_{\mathbf{i}} \operatorname{sgn}(\mathbf{i})
  f(\Omega(a_{i_1})\wedge \cdots \wedge \Omega(a_{i_{|f|}}) )
  g(\Omega(a_{i_{|f|+1}})\wedge \cdots \wedge \Omega(a_{i_{|f|+|g|}})
  ) \end{multline*} where $f \in \Alt_A^{|f|}(\Omega(A),A)$, $g \in
\Alt_A^{|g|}(\Omega(A),A)$, the sum is over all $\mathbf{i} =
(i_1,\ldots,i_{|f|+|g|})$ such that $i_1<\cdots<i_{|f|}$ and
$i_{|f|+1}< \cdots <i_{|f|+|g|}$, and $\operatorname{sgn}(\mathbf{i})$
is the sign of the permutation $r \mapsto i_r$.  The shuffle product
on the space of alternating forms is obtained by transferring the
natural multiplication on $A \otimes_k \Lambda^*(V^*)$ through the
isomorphisms
\begin{equation*}
\Alt_A(\Omega(A),A) \cong \Hom_A(A\otimes
\Lambda^*(V),A) \cong \Hom_k(\Lambda^*(V),A) \cong A\otimes_k
\Lambda^*(V^*). \end{equation*}

\begin{defn} The \textbf{Poisson cohomology} $\HP^*(A)$ of a Poisson
  algebra $A$ is the cohomology of the differential graded algebra
  $\operatorname{Alt}_A^*(\Omega(A),A)$.  \end{defn} If $A$ is a
graded algebra and its Poisson bracket respects the grading then the
Poisson cohomology groups are bigraded; in this case we write
$\HP^{ij}(A)$ for the part in homological degree $i$ and internal
degree $j$.

When the K\"ahler differentials are projective as an $A$-module, so in
particular when $A$ is polynomial,  $P(A) \otimes_A \Lambda
^*_A(\Omega(A))$ is a projective resolution of $A$ as
a $P(A)$-module and the Poisson cohomology of $A$ is
isomorphic as an algebra to $\ext_{P(A)}^*(A,A)$ \cite[p.81]{Hueb}.

\section{The semiclassical limit and $q$-deformations} \label{scl}
Let $\mathcal{A}$ be a 
$\kq$-algebra which is a
torsion-free $\kq$-module and suppose $A =
\mathcal{A}/(q-1)\mathcal{A}$ is commutative.  If $u \in \cA$ write
$\bar{u}$ for its image in $A$.  Then $A$ is a Poisson
algebra with bracket
\[ \{\bar{a},\bar{b} \} := \overline {\beta(a,b) }\]
where $\beta(a,b)$ is the unique element of
$\mathcal{A}$ such that $ ab-ba = (q-1)\beta(a,b)$.
\begin{defn} \label{def_scl} With $\cA$, $\{-,-\}$ and $A$ as above we say that  $A$
  is the semiclassical limit of $\mathcal{A}$. \end{defn}
See for example \cite{Dumas, Goodearl}.

\begin{defn} A $\kq$-subalgebra $\mathcal{R}$ of a $k(q)$-algebra
  $\sR$ is called a $\kq$-form of
  $\sR$ if the natural map $\mathcal{R}\otimes_\kq k(q) \to \sR$ is an
  isomorphism.
We say
  $\sR$ is a $q$-deformation
  of the $k$-algebra $R$ via the $\kq$-form $\mathcal{R}$  if
 $\mathcal{R}
   /(q-1)\mathcal{R}$ is isomorphic as a $k$-algebra to  $R$.  \end{defn}

 We will use an analogous notion of $q$-deformation for differential
 graded algebras, obtained by replacing `algebra' and `subalgebra' by
 `DG-algebra' and `sub-DG-algebra' everywhere in the above
 definition. Note that if $\sR$ is (bi)graded then so is $R$, and they
 have the same Hilbert series.

 Suppose $\cA$ is a $\kq$-form of a $k(q)$-algebra $\sA$ and that
 $\cA$ has a semiclassical limit $A$.  We want to relate the Poisson
 enveloping algebra $P(A)$ and the enveloping algebra $\sA^e = \sA
 \otimes_{k(q)} \sA^{\opp}$ of $\sA$ using the notion of
 $q$-deformation.  The subalgebra of $\sA^e$ generated by
 $\mathcal{A}\otimes 1$ and $1\otimes \mathcal{A}^{\opp}$ is not
 suitable as a $k[q^{\pm 1}]$-form because it is commutative modulo $q-1$,
 and $P(A)$ is in general noncommutative.  To help us define a
 suitable $\kq$-form we need some special elements of $\sA^e$.

\begin{defn} Let $\sa \in \sA$.  Then $\tilde\Omega(\sa)$ is
\[ \frac{\sa\otimes 1 - 1 \otimes \sa}{q-1} \in \sA^e. \]
\end{defn}

\begin{lem} \label{modified_rels}
Let $x,y \in \mathcal{A}$.  Write $\sx$ and $\sy$ for the elements
$x\otimes 1$ and $y \otimes 1$ of $\sA^e$.  Then in
  $\sA^e$, 
  \begin{align}
\label{q_deriv} \tilde \Omega(xy) & = \sx \tilde\Omega(y) + \sy
\tilde\Omega(x) + (q-1)\tilde\Omega(y)\tilde\Omega(x) \\
\label{rrel2} \tilde\Omega(x)\sy - \sy\tilde\Omega(x) &= \beta(x,y)\otimes 1 \\
\label{rrel3} \tilde\Omega(x)
\tilde\Omega(y)-\tilde\Omega(x)\tilde\Omega(y) & =
\tilde\Omega(\beta(x,y)). \end{align}
\end{lem}

The proof is a simple computation.  This lemma shows that the $\tilde\Omega(\sa)$ behave like
$q$-analogues of the K\"ahler differentials of $A$ --- compare
(\ref{q_deriv}) with (\ref{k_diffs}) and (\ref{rrel2}), (\ref{rrel3})
with (\ref{rel2}), (\ref{rel3}). 

\begin{rem}
If $\sA$ is a quadratic algebra generated by homogeneous
elements $\sx_i$ of degree one with a basis of the form
$\sx_1^{a_1}\cdots \sx_n^{a_n}$ for $a_i \geq 0$, these relations,
together with the relations of $\sA$, are enough to give a
presentation of $\sA^e$.  This follows by counting the dimension of
the spaces of relations. 
\end{rem}

For the rest of this section we assume that $\sA$ is graded, that
$\sx_1,\ldots,\sx_n$ is a homogeneous generating set for $\sA$, and
that $A$ is polynomial on the images $\overline{\sx_i}$ of the
$\sx_i$.  Let $y_i$ be the generator of $P(A)$ corresponding to
$\overline{\sx_i}$; then $P(A)$ is graded by putting $y_i$ in the same
degree as $\sx_i$.

\begin{lem} \label{surject}  Let $\cA'$ be the $\kq$-subalgebra of
  $\sA^e$ generated by $\sx_i\otimes 1$ and $\tilde\Omega(\sx_i)$ for
  $1 \leq i \leq n$.   Then there is a surjection of
  graded algebras $P(A) \twoheadrightarrow \mathcal{A}'/(q-1)\cA'$ defined by $y_i \mapsto
  \overline{\sx_i \otimes 1}$ and $\Omega(y_i) \mapsto
\overline{\tilde\Omega(\sx_i)}$.
\end{lem}

\begin{proof}
We must show that if we substitute $\sx_i \otimes 1$
for $y_i$ and $\tilde\Omega(\sx_i)$ for $\Omega(y_i)$
in the defining relations   (\ref{rel1}), (\ref{rel2}), (\ref{rel3})
of $P(A)$, the resulting expressions lie in  $(q-1)\mathcal{A}'$.  This is true of  the relations (\ref{rel1})
because the semiclassical limit exists.

\noindent \textbf{(\ref{rel2}):} we need
\[ \tilde\Omega(\sx_i)(\sx_j \otimes 1) -
(\sx_j \otimes 1){\tilde\Omega(\sx_i)} =
{\beta(\sx_i,\sx_j)\otimes 1}  \mod (q-1)\]
but this is immediate from (\ref{rrel2}).

\noindent \textbf{(\ref{rel3}):} Suppose first that
$\beta(\sx_i,\sx_j) = x_{i_1}\cdots x_{i_N}$ is a monomial.  We need
\[
{\tilde\Omega(\sx_i)}{\tilde\Omega(\sx_j)}-{\tilde\Omega(\sx_j)}{\tilde\Omega(\sx_i)}
= \sum_j( {\sx_{i_1}\cdots \widehat{\sx_{i_j}} \cdots
  \sx_{i_N}\otimes 1}){\tilde\Omega(\sx_{i_j})} \mod (q-1)\]
but this follows easily by induction on $N$ using (\ref{q_deriv}).  The general case, when $\beta(\sx_i,\sx_j)$ is not
assumed to be a monomial, follows by linearity.
\end{proof}

\begin{lem}\label{annoy}
Let $\sz_1,\ldots,\sz_n$ be a homogeneous generating set for a graded $k(q)$-algebra
$\sB$ and let $\mathcal{B}$ be the $k[q^{ \pm 1 }]$-subalgebra of $\sB$
generated by $\sz_1,\ldots,\sz_n$.  Then $\mathcal{B}\otimes_\kq k(q)
\cong \sB$ as graded algebras.
\end{lem}

\begin{proof}
  The assignment $\sz_i \mapsto \sz_i \otimes 1$ determines a graded
  surjective homomorphism of algebras $\sB \to \mathcal{B}\otimes_\kq
  k(q)$.  Suppose the kernel contains a non-zero element
  $X = \sum_{\mathbf{i}} (a_{\mathbf{i}} /
  b_{\mathbf{i}})\sz_{\mathbf{i}}$ where $a_{\mathbf{i}},
  b_{\mathbf{i}} \in \kq$ and $\sz_{\mathbf{i}} =
  \sz_{i_1},\ldots,\sz_{i_N}$.  We may assume all the $b_{\mathbf{i}}$
  are equal to $1$ by multiplying by an appropriate element of $\kq$,
  so that $X\in \mathcal{B}$.  Now $\mathcal{B}$ is a free
  $k[q^{ \pm 1 }]$-module as it is the direct sum of its graded pieces which
  are finitely generated $\kq$-submodules of a $k(q)$-module and
  therefore torsion-free. Therefore the  map $\mathcal{B} \to \mathcal{B}
  \otimes _\kq k(q)$ is injective and so $X=0$, a contradiction.
\end{proof}

It follows that $\cA'$ is a $\kq$-form of $\sA^e$.

\begin{cor} \label{useful_cor} If $\dim_kP(A)_m= \dim_{k(q)} \sA^e_m$
  for all $m$  then $\sA^e$ is a $q$-deformation of $P(A)$ via the
  $\kq$-form $\cA'$. \end{cor}

\begin{proof}
If we regard $k$ as a $k[q^{ \pm 1 }]$-module with $q$ acting
as $1$ then \[\mathcal{A}'/(q-1)\mathcal{A}'\cong\mathcal{A}'\otimes_\kq
k\] as graded algebras. 
Each graded piece of $\cA'$ is a free $\kq$-module, since it embeds
into a $k(q)$-module $\sA^e$ and is therefore torsion free.  So
for any $m$,
\[  \dim_k (\cA' \otimes _\kq k)_m = \dim_{k(q)} (\cA'\otimes_\kq
k(q))_m = \dim_{k(q)} \sA^e_m = \dim_k P(A)_m\]
where the second equality is because of Lemma \ref{annoy}.  It follows
that the surjection of Lemma \ref{surject} has to be injective.
\end{proof}

When $\sA$ has a PBW basis of polynomial type we can apply the
previous corollary to get:

\begin{theo} \label{poly_pbw} Suppose the set of elements of the form
  $\sx_1^{a_1}\cdots \sx_n^{a_n}$ for $a_i\geq 0$ form a basis of $\sA$.
  Then $\sA^e$ is a $q$-deformation of $P(A)$ via the $\kq$-form
  $\cA'$.  \end{theo}

\begin{proof}
Lemma \ref{poisson_pbw} says that $P(A)$
  has a PBW basis consisting of the elements
\[ y_1^{a_1}\cdots y_n^{a_n}\Omega(y_1)^{b_1}\cdots
\Omega(y_n)^{b_n} \]
for $a_i,b_i \geq 0$, where $y_i$ is the generator of $P(A)$ corresponding to
$\overline{\sx_i}$.  
Our assumptions on $\sA$ mean that $\sA^e$ has a
PBW basis consisting of
\[ (\sx_1\otimes 1)^{a_1}\cdots (\sx_n\otimes 1)^{a_n} (1\otimes
\sx_1)^{b_1}\cdots (1\otimes \sx_n)^{b_n}\]
for $a_i,b_i \geq 0$, therefore the hypothesis of Corollary
\ref{useful_cor} holds.
\end{proof}

\subsection{Example: the quantum plane} \label{quant_aff_sp_ex}
\begin{defn} The coordinate ring of the quantum plane is the
  $k(q)$-algebra  $\sA$  generated by $\sx,\sy$ subject
  to the relation 
$ \sx \sy = q \sy \sx $.
\end{defn}
$\sA$ is graded with $|\sx|=|\sy|=1$; it is a quadratic
algebra with basis consisting of all elements
$\sx^{a}\sy ^{b} $
with $a,b \geq 0$.  The $\kq$-form $\cA$ generated
by $\sx,\sy$ has semiclassical limit $A$ which is  polynomial
on the images $x,y$ of $\sx,\sy$ and has Poisson structure
determined by
$\{ x, y\} = x y$.
By Lemma \ref{pea} the Poisson enveloping algebra $P(A)$ is
generated by $x,y$ and $\Omega(x),\Omega(y)$  subject to 
\begin{gather*}
x y = y x,  
\Omega(x)y - y \Omega(x) = xy, \Omega(y)x-x\Omega(y)=-xy\\
\Omega(x)\Omega(y)-\Omega(y)\Omega(x) =   x\Omega(y)+
y\Omega(x)
\end{gather*}

Theorem \ref{poly_pbw} shows $\sA^e$ is a $q$-deformation of $P(A)$ via
the $\kq$-form generated by  $\sx,\sy$ and
$\tilde\Omega(\sx),\tilde\Omega(\sy)$.
By Lemma \ref{modified_rels}, the following relations hold in
$\sA^e$:
\begin{gather*}
\sx \sy = q\sy \sx , \tilde\Omega(\sx)\sy - \sy \tilde\Omega(\sx) = \sx \sy,
\tilde\Omega(\sy)\sx-\sx\tilde\Omega(\sy)=-\sx\sy\\
q\tilde\Omega(\sx)\tilde\Omega(\sy)-\tilde\Omega(\sx)\tilde\Omega(\sy) =
\sx\tilde\Omega(\sy)+
\sy\tilde\Omega(\sx). 
\end{gather*}
In fact by counting the dimension of the relation space they are sufficient to give a presentation of $\sA^e$.

\subsection{Example: $2\times 2$ quantum matrices}
\begin{defn} The algebra of $2\times 2$ quantum matrices $\sM$ is
  the $k(q)$-algebra generated by $\sa,\bb,\scc,\sd$ subject to the
  relations
\begin{gather} \label{m2rels}
\sa \bb = q \bb \sa \,\,\,\,\,\, \sa \scc = q \scc \sa\,\,\, \,\,\, \bb\scc=\scc \bb\,\,\, \,\,\, \bb
\sd = q \sd \bb\,\,\, \,\,\, \scc \sd = q \sd \scc \\
\nonumber \sa \sd - \sd \sa = (q-q^{-1})\bb \scc. \end{gather}
\end{defn}
$\sM$ is graded with $\sa,\bb,\scc,\sd$ in degree $1$; it is a
quadratic algebra which admits a 
basis consisting of all elements 
$ \sa^i \bb^j \scc ^k \sd^l $
for $i,j,k,l\geq 0$. %\cite{JFJFJF} 

The $\kq$-subalgebra $\mathcal{M}$  generated by
$\sa,\bb,\scc,\sd$ has semiclassical limit $M$ which is polynomial
on the images $a,b,c,d$ of $\sa,\bb,\scc,\sd$ with Poisson structure
determined by
\begin{gather}\label{bracket_formulas}
  \{a,b\} = ab \,\,\,\,\,\, \{a,c\} = ac \,\,\,\,\,\, \{a,d\} = 2bc
  \,\,\,\,\,\, \{b,c\}=0 \,\,\,\,\,\,
  \{b,d\}=bd\,\,\,\,\,\,\{c,d\}=cd. \end{gather}
By Lemma \ref{pea} the Poisson enveloping algebra $P(M)$ is
generated by $a,b,c,d,\Omega(a),\Omega(b),\Omega(c),\Omega(d)$ subject
to relations saying that $a,b,c,d$ commute, and 
\begin{center}\begin{tabular}{ll}
$\Omega(a)a - a\Omega(a)=0$ &
$\Omega(a)b-b\Omega(a) = ab$ \\
$\Omega(a)c-c\Omega(a) = ac $&
$\Omega(a)d-d\Omega(a) = 2bc $\\
$\Omega(b)a-a\Omega(b) = -ab $&
$\Omega(b)b-b\Omega(b)=0 $\\
$\Omega(b)c-c\Omega(b) = 0 $&
$\Omega(b)d-d\Omega(b) = bd $\\
$\Omega(c)a-a\Omega(c) = -ac $&
$\Omega(c)b-b\Omega(c) = 0 $\\
$\Omega(c)c-c\Omega(c)=0 $&
$\Omega(c)d-d\Omega(c) = -cd $\\
$\Omega(d)a-a\Omega(d) = -2bc $&
$\Omega(d)b-b\Omega(d) = -bd $\\
$\Omega(d)c-c\Omega(d) = -cd $&
$\Omega(d)d-d\Omega(d) =0$\\
\multicolumn{2}{l}{$\Omega(a)\Omega(b)-\Omega(b)\Omega(a) =
a\Omega(b)+b\Omega(a)$} \\
\multicolumn{2}{l}{$\Omega(a)\Omega(c)-\Omega(c)\Omega(a) =
a\Omega(c)+c\Omega(a)$} \\
\multicolumn{2}{l}{$\Omega(a)\Omega(d)-\Omega(d)\Omega(a) =
2b\Omega(c)+2c\Omega(b) $}\\
\multicolumn{2}{l}{$\Omega(b)\Omega(c)-\Omega(c)\Omega(b) = 0 $}\\
\multicolumn{2}{l}{$\Omega(b)\Omega(d)-\Omega(d)\Omega(b) =
b\Omega(d)+d\Omega(b)$ }\\
\multicolumn{2}{l}{$\Omega(c)\Omega(d)-\Omega(d)\Omega(c) =
c\Omega(d)+d\Omega(c)$} \\
\end{tabular}
\end{center}
Theorem \ref{poly_pbw} shows $\sM^e$ is a
$q$-deformation of $P(M)$ via the $\kq$-form generated by
$\sa,\bb,\scc,\sd$ and $\tilde\Omega(\sa), \tilde\Omega(\bb), \tilde\Omega(\scc),
\tilde\Omega(\sd)$.

By Lemma \ref{modified_rels}, 
the following relations hold in $\sM^e$:
\begin{center}\begin{tabular}{ll}
$\tilde\Omega(\sa)\sa - \sa\tilde\Omega(\sa)=0 $&
$\tilde\Omega(\sa)\bb-\bb\tilde\Omega(\sa) = \sa\bb$ \\
$\tilde\Omega(\sa)\scc-\scc\tilde\Omega(\sa) = \sa\scc $&
$\tilde\Omega(\sa)\sd-\sd\tilde\Omega(\sa) = (1+q^{-1})\bb\scc$ \\
$\tilde\Omega(\bb)\sa-\sa\tilde\Omega(\bb) = -\sa\bb$ &
$\tilde\Omega(\bb)\bb -\bb\tilde\Omega(\bb)=0 $\\
$\tilde\Omega(\bb)\scc-\scc\tilde\Omega(\bb) = 0 $&
$\tilde\Omega(\bb)\sd-\sd\tilde\Omega(\bb) = \bb\sd $\\
$\tilde\Omega(\scc)\sa-\sa\tilde\Omega(\scc) = -\sa\scc $&
$\tilde\Omega(\scc)\bb-\bb\tilde\Omega(\scc) = 0 $\\
$\tilde\Omega(\scc)\scc -\scc\tilde\Omega(\scc) = 0$ &
$\tilde\Omega(\scc)\sd-\sd\tilde\Omega(\scc) = -\scc\sd$ \\
$\tilde\Omega(\sd)\sa-\sa\tilde\Omega(\sd) = -(1+q^{-1})\bb\scc$ &
$\tilde\Omega(\sd)\bb-\bb\tilde\Omega(\sd) = -\bb\sd$ \\
$\tilde\Omega(\sd)\scc-\scc\tilde\Omega(\sd) = -\scc\sd$ &
$\tilde\Omega(\sd)\sd -\sd\tilde\Omega(\sd)=0$ \\
\multicolumn{2}{l}{$q\tilde\Omega(\sa)\tilde\Omega(\bb)-\tilde\Omega(\bb)\tilde\Omega(\sa)
= \sa\tilde\Omega(\bb)+\bb\tilde\Omega(\sa)$} \\
\multicolumn{2}{l}{$q\tilde\Omega(\sa)\tilde\Omega(\scc)-\tilde\Omega(\scc)\tilde\Omega(\sa)
= \sa\tilde\Omega(\scc)+\scc\tilde\Omega(\sa)$} \\
\multicolumn{2}{l}{$\tilde\Omega(\sa)\tilde\Omega(\sd)-\tilde\Omega(\sd)\tilde\Omega(\sa)
  = (1+q^{-1})\bb\tilde\Omega(\scc)+(1+q^{-1})\scc\tilde\Omega(\bb)
  -(q-q^{-1})\tilde\Omega(\bb)\tilde\Omega(\scc)$}\\
\multicolumn{2}{l}{$\tilde\Omega(\bb)\tilde\Omega(\scc)-\tilde\Omega(\scc)\tilde\Omega(\bb)
= 0 $}\\
\multicolumn{2}{l}{$q\tilde\Omega(\bb)\tilde\Omega(\sd)-\tilde\Omega(\sd)\tilde\Omega(\bb)
= \bb\tilde\Omega(\sd)+\sd\tilde\Omega(\bb)$} \\
\multicolumn{2}{l}{$q\tilde\Omega(\scc)\tilde\Omega(\sd)-\tilde\Omega(\sd)\tilde\Omega(\scc)
= \scc\tilde\Omega(\sd)+\sd\tilde\Omega(\scc)$} \\
\end{tabular}\end{center}

\section{Koszul algebras and modules}

We refer to \cite{PP} for general background on quadratic and Koszul
algebras.  Recall that a graded $k$-algebra $\Lambda$ 
is called \textbf{quadratic} if it there is a finite-dimensional
vector spce $V$ and a subspace $R \leq V\otimes_k V$ such that
$\Lambda \cong T(V)/(R)$ where $T(V)$ denotes the tensor algebra.  The
\textbf{quadratic dual} $\Lambda ^!$ of $\Lambda$ is
$T(V^*)/(R^\perp)$ where $V^* = \Hom_k(V,k)$ and $R^\perp$ is the
image of the annihilator of $R$ under the canonical isomorphism of
$(V\otimes_k V)^*$ with $V^* \otimes_k V^*$.  

\begin{ex} \label{quad_dual_2x2_matrices}
  The quadratic dual $\sM^!$ of $\sM$ is generated by
  $\sa^*,\bb^*,\scc^*, \sd^*$ subject to
  \begin{gather*}
    \sa^{*2}, \bb^{*2}, \scc^{*2}, \sd^{*2}, 
    \bb^* \scc^* + \scc^* \bb^* + (q-q^{-1})\sa^* \sd^* \\
    q\sa^* \bb^* +\bb^* \sa^*, q\sa^* \scc^* + \scc^* \sa^*, \sa^* \sd^*
    + \sd^* \sa^*, q\bb^* \sd^* + \sd^* \bb^*, q\scc^* \sd^* + \sd^*
    \scc^* .
  \end{gather*}
\end{ex}

A graded left
$\Lambda$-module $M$ is called \textbf{quadratic} if it is isomorphic
as a graded module to one of the form $(\Lambda \otimes_k M_0)/\Lambda
H$ where $M_0$ is a finite-dimensional vector space, $H \leq \Lambda_1
\otimes M_0$, and $M_0$ is homogeneous with respect to the grading.
The quadratic dual $M^!$ of $M$ is the left $\Lambda^!$-module
$(\Lambda^!  \otimes_k M_0^*)/(\Lambda^!  H^\perp)$ where $H^\perp$
denotes the image of the annihilator of $H$ under the canonical
isomorphism between $(\Lambda_1 \otimes_k M_0)$ and $\Lambda_1^*
\otimes_k M_0 ^* = (\Lambda^!)_1 \otimes_k M_0^*$.  If $M$ and $N$ are
graded left $\Lambda$-modules, $\ext^* _\Lambda(M,N)$ is bigraded, by
homological degree and by internal degree.  We write
$\ext^{ij}_\Lambda (M,N)$ for the part with homological degree $i$ and
internal degree $j$.  Write $k$ for the trivial $\Lambda$-module
$\Lambda / \bigoplus_{i>0}\Lambda_i$.  A quadratic algebra $\Lambda$
is called \textbf{Koszul} if $\ext^{ij}_\Lambda(k,k)$ is zero whenever
$i\neq j$, in which case $\ext^*_\Lambda(k,k) \cong \Lambda^!$ as
algebras.  A graded module $M$ over a Koszul algebra $\Lambda$ is
called Koszul if $\ext^{ij}_\Lambda(M,k)$ is zero if $i \neq j$, or
equivalently if $M$ admits a linear projective resolution, that is, a
projective resolution $P_* \twoheadrightarrow M$ such that $P_i$ is
generated by its component of degree $i$.

\subsection{Hochschild cohomology of Koszul algebras}

If $\Lambda$ is Koszul then $\Lambda^e$ is Koszul:
$\Lambda^{\operatorname{op}}$ is Koszul by \cite[remark on p.20]{PP},
and tensor products of Koszul algebras are Koszul by \cite[Corollary
3.1.2]{PP}.  Furthermore $\Lambda$ is a Koszul $\Lambda^e$-module by
\cite[Corollary 2.2]{GHMS}.

\begin{lem} \label{kos_dual} The quadratic dual $(_{\Lambda^e}\Lambda)^!$ of the
  $\Lambda^e$-module $\Lambda$ is isomorphic as a vector space to the
  dual quadratic algebra $\Lambda^!$.
\end{lem}

\begin{proof}
Let $\Lambda = T(V)/(R)$ so that $\Lambda^e =T(V \oplus V')/(R\oplus
R' \oplus C)$ where $V'$ is isomorphic to $V$ via a map that sends
$v\in V$
to $v'$, $R'$ is the image of $R$ under the twist map $\tau : v
\otimes w \mapsto w' \otimes v'$ and $C$ has basis $x_i\otimes x_j'-
x_j'\otimes x_i$ where $x_1,\ldots,x_n$ is some fixed basis of $V$.

Thus $(\Lambda^e)^!$ is isomorphic to $T(V^* \oplus V'^*)/(R^\perp
\oplus R'^\perp \oplus D)$ where $D$ has a basis consisting of all
tensors of the form $x_i^* \otimes x_j'^* + x_j'^* \otimes x_i ^*$
where the $x_i^*$ are the basis of $V^*$ dual to $x_1,\ldots,x_n$.  It
follows $(\Lambda ^e)^! \cong \Lambda^! \hat\otimes_k
(\Lambda^!)^{\textrm{op}}$ where $\hat\otimes_k$ denotes the graded commutative
tensor product: $(\lambda_1 \hat\otimes \mu_1)\cdot (\lambda_2
\hat\otimes \mu_2) = (-1)^{mn}\lambda_1\lambda_2 \hat\otimes
\mu_1\mu_2$ for $\lambda_2$, $\mu_1$ homogeneous of degrees $m$ and
$n$ respectively.

As a $\Lambda^e=T(V\oplus V')/(R\oplus R' \oplus C)$-module, $\Lambda
= \Lambda^e / (x_i - x_i' : i=1,\ldots,n)$.  Therefore the quadratic
dual of $_{\Lambda^e}\Lambda$ is $(\Lambda^e)^!/(x_i^* + x_i'^* :
i=1,\ldots,n)$.  This ideal corresponds to the ideal $(x_i^*\otimes 1
+ 1\otimes x_i'^*)$ under the isomorphism between $T(V\oplus
V')/(R\oplus R' \oplus C)$ and $\Lambda^! \hat\otimes_k
(\Lambda^!)^{\textrm{op}}$.  There is an exact sequence of $\Lambda^!
\hat\otimes (\Lambda^!)^{\textrm{op}}$-modules
\[ 0 \to (x_i^*\otimes 1 + 1\otimes x_i'^*) \to \Lambda^!
\hat\otimes_k (\Lambda^!)^{\textrm{op}} \to \hat\Lambda ^! \to 0 \]
where $\hat\Lambda^!$ is the $\Lambda^! \hat\otimes_k
(\Lambda^!)^{\textrm{op}}$-module which is $\Lambda^!$ as a vector
space, and with action $(\lambda \hat\otimes \mu') \cdot x =
(-1)^{|\mu||x| + |\mu|(|\mu|+1)/2}\lambda x \mu$ for $\lambda, x,\mu
\in \Lambda^!$, homogeneous, and the map $\Lambda^! \hat\otimes_k
(\Lambda^!)^{\textrm{op}} \to \hat\Lambda^!$ is determined by
$1\otimes 1 \mapsto 1$.  This completes the proof.
\end{proof}

\begin{defn}
  Let $M$ be a Koszul left-module for the Koszul $k$-algebra
  $\Gamma=T(V)/(R)$ and let $e_\Gamma \in \Gamma \otimes_k \Gamma^!$
  be $\sum_i v_i \otimes v_i^*$ where $v_i$ runs over a basis of $V$
  and $v_i^*$ is the corresponding dual basis element of $V^* =
  \Gamma^!_1$.  Then the \textbf{Koszul resolution} $K_\Gamma(M)$ is
  $\Gamma \otimes_k (M^!)^*$ with differential given by
  right-multiplication by $e_\Gamma$, and the \textbf{Koszul
    cocomplex} $\bar{K}_\Gamma(M)$ is $M \otimes_k M^!$ with
  differential given by left-multiplication by $e_\Gamma$
\end{defn}
\cite[\S2.3]{PP} shows that $K_\Gamma(M)$ is a minimal free resolution
of $M$, so that $\ext_\Gamma^*(M,M)$ is computed by cohomology of the
cocomplexes
\begin{equation*} %\label{kcmp}
\Hom_\Gamma( \Gamma \otimes_k (M^!)^*, M ) \cong \Hom_k((M^!)^*,M)
\cong M \otimes_k M^! = \bar{K}_\Gamma(M).
\end{equation*}

Consider the special case when $\Gamma
= \Lambda^e$ for some Koszul algebra $\Lambda$ and $M =
{_{\Lambda^e}\Lambda}$.  By Lemma \ref{kos_dual},
$(_{\Lambda^e}\Lambda)^!$ is $\hat \Lambda ^!$ so
$\bar{K}_{\Lambda^e}(\Lambda)$ is isomorphic to $\Lambda \otimes_k
\Lambda ^!$ as a vector space.  The corresponding differential is
\begin{equation}\label{diff} 
  \lambda \otimes \mu \mapsto \sum_i ( x_i \lambda \otimes x_i^* \mu
  +(-1)^{|\mu|+1} \lambda x_i \otimes \mu x_i^*) 
\end{equation}
for $\mu$ homogeneous of degree $|\mu|$ -- see \cite[p.5]{vdb}.
This differential makes $\Lambda \otimes_k \Lambda^!$, with its
natural multiplication, into a differential graded algebra.

This means that the
Hochschild cohomology ring $\HH(\Lambda) =
\ext^*_{\Lambda^e}(\Lambda,\Lambda)$ has two multiplications: one it
inherits as the cohomology of the differential graded algebra $\Lambda
\otimes_k \Lambda ^!$ and one from the Yoneda product on Ext.  We want
to show that they agree.

\begin{lem} \label{dga_for_HH}The cohomology of the differential graded algebra $\Lambda
  \otimes_k \Lambda^!$ equipped with the differential (\ref{diff}) is
  isomorphic as an algebra to $\HH^*(\Lambda)$.
\end{lem}

\begin{proof}
We already know that the cohomology of $\Lambda
  \otimes_k \Lambda^!$ with this differential agrees with
  $\HH^*(\Lambda)$ as a vector space.  To show that the two products
  are the same we need to examine the Koszul resolution more closely.
If $\Lambda = T(V)/(R)$ then $\Lambda^!_m$ is equal to 
\[ \frac{(V^*)^{\otimes m} }{ \sum_{i=0}^{m-2} (V^*)^{\otimes i}
  \otimes R^\perp \otimes (V^*)^{\otimes (m-2-i)} } \]
(the denominator is to be interpreted as zero if $m=0$ or $1$).  For
$m \geq 2$ the denominator is the annihilator of 
\[N_m= \bigcap_{i=0}^{m-2} V^{\otimes i} \otimes R \otimes V^{\otimes(m-2-i)} \]
so we can identify $(\Lambda^!_m)^*$ with $N_m$, and the Koszul resolution
$K_{\Lambda^e}(\Lambda)$ of $\Lambda$ over
$\Lambda^e$ can be written as
\begin{equation}\label{kos} 
K_{\Lambda^e}(\Lambda)_m = \Lambda \otimes N_m \otimes \Lambda 
\end{equation}
where $N_1=V$ and $N_0=k$.
As before let $x_1,\ldots,x_n$ be a basis of $V$, and given a sequence $\mathbf{i}=(i_1,\ldots,i_m)$
write $x_{\mathbf{i}}$ for $x_{i_1} \otimes \cdots \otimes x_{i_m} \in
V^{\otimes m}$.  The differential on (\ref{kos}) is
\[ 1 \otimes \left(\sum_\mathbf{i} \alpha_{\mathbf{i}} x_{\mathbf{i}} \right) \otimes 1
\mapsto \sum_{\mathbf{i}} \alpha_{\mathbf{i} } ( x_{i_1} \otimes
x_{(i_2,\ldots,i_m)} \otimes 1 - 1 \otimes x_{(i_1,\ldots,i_{m-1})}
\otimes x_{i_m} ) \]
where $\alpha_{\mathbf{i}} \in k$. 

Write $\mathbb{B}_*$ for the standard (bar) complex of $\Lambda$ \cite[IX.6]{CE}, whose $m$th
term is $\mathbb{B}_m= \Lambda^{\otimes (m+2)}$.  The inclusion $V
\hookrightarrow \Lambda$ induces a map
\[ \iota : K_{\Lambda^e}(\Lambda)_* \to \mathbb{B}_* \]
which is a morphism of chain complexes \cite[Proposition 3.3]{vdb}.
The bar complex of $\Lambda$ admits a comultiplication $\Delta:
\mathbb{B}_* \to \mathbb{B}\otimes_\Lambda \mathbb{B}$ defined by
\[ 
\Delta( \lambda \otimes y_{\mathbf{i}} \otimes \mu) = \sum_r
(\lambda \otimes y_{(i_1,\ldots,i_r)} \otimes 1) \otimes (1 \otimes
y_{(i_{m+1}, \ldots, i_m)} \otimes \mu) 
\]
where $y_{\mathbf{i}} = y_{i_1}\otimes\cdots\otimes y_{i_m} \in
\Lambda^{\otimes m}$. We will show that 
\begin{equation}\label{comult} 
\Delta (\im \iota) \subseteq (\im \iota )\otimes_\Lambda (\im \iota)
\end{equation}
so that $\Delta$ induces a comultiplication on $K_{\Lambda^e}(\Lambda)$.  
This is equivalent to proving that if 
$\sum_{\mathbf{i}} \alpha_{\mathbf{i}} x_{\mathbf{i}} \in N_m$ then 
\[
\sum _{\mathbf{i}}\alpha_{\mathbf{i}} x_{(i_1,\ldots,i_r)} \otimes
x_{(i_{r+1},\ldots,i_m)} \in N_r \otimes N_{m-r} 
\]
for any $r\leq m$.  Corollary 3.3 of \cite{Priddy} says that for any
sequence $\mathbf{j}$ of length $m-r$, 
\begin{equation*} %\label{cut_right} 
  \sum _{\mathbf{i} :
    (i_{r+1},\ldots,i_m)=\mathbf{r} } \alpha_{\mathbf{i}}
  x_{(i_1,\ldots, i_{r})} \in N_{r} \end{equation*} Thus
\[\sum _{\mathbf{i}}\alpha_{\mathbf{i}}  x_{(i_1,\ldots,i_r)} \otimes
x_{(i_{r+1},\ldots,i_m)} =\sum_{\mathbf{j}} \left( \sum_{\mathbf{i} :
    (i_{r+1},\ldots,i_m)=\mathbf{r} }
  \alpha_{\mathbf{i}} x_{(i_1,\ldots,i_r)} \right) \otimes
x_{\mathbf{j}}  \in N_r \otimes V^{\otimes m-r}\]
Similarly it lies in $V^{\otimes r} \otimes N_{m-r}$, so in 
\[(N_r \otimes V^{\otimes m-r}) \cap ( V^{\otimes r} \otimes
N_{m-r})= N_r \otimes N_{m-r} \]
completing the proof of (\ref{comult}).

If we identify $\Lambda \otimes_\Lambda \Lambda$ with $\Lambda$ then
$\mathbb{B}_*\otimes_\Lambda \mathbb{B}$ is a free resolution of
$\Lambda$ and $\Delta$ is a chain map lifting the identity map on
$\Lambda$.  Because of (\ref{comult}) the same holds for
$K_{\Lambda^e}(\Lambda)$.  \cite[p.4]{BGSS} point out that the Yoneda
product on $\HH(\Lambda)$ can be computed as follows: if $f:K_{\Lambda^e}(\Lambda)_r \to \Lambda $ and
$g:K_{\Lambda^e}(\Lambda)_s \to \Lambda$ are cocycles, the product of
the cohomology elements they represent is represented by 
 $f*g =(f\otimes_ \Lambda g)\circ  \Delta$
where $f\otimes_\Lambda g: K_{\Lambda^e}(\Lambda)_r \otimes_\Lambda K_{\Lambda^e}(\Lambda)_s
\to \Lambda$ is $\kappa \otimes \kappa' \mapsto f(\kappa)g(\kappa')$.

There is a linear isomorphism $\phi: \Lambda \otimes \Lambda^! \mapsto \Hom
_{\Lambda^e}(K_{\Lambda^e}(\Lambda)_*, \Lambda)$ that sends $\lambda \otimes
\mu$ to 
\[ 1\otimes \mathbf{n}\otimes 1 \mapsto \langle \mu,\mathbf{n} \rangle
\lambda\]
where $\mu \in \Lambda^!_m$, $\mathbf{n} \in N_m$ and $\langle
-,-\rangle$ is the pairing between $\Lambda^!_m$ and $N_m$.  We will
show this is is a homomorphism of algebras 
when $ \Hom_{\Lambda^e}(K_{\Lambda^e}(\Lambda)_*, \Lambda)$ is equipped with the
product $*$.

Since the product in $\Lambda^!$ is induced by tensor multiplication
in $T(V^*)$, if $\mu = \hat{\mu} + \ann(N_r) \in \Lambda^!_r$ and $\mu'
= \hat{\mu'} + \ann(N_s) \in \Lambda^!_s$ then
\[\phi(\lambda \lambda' \otimes \mu \mu') (1 \otimes \left(\sum_\mathbf{i}
\alpha_\mathbf{i} x_{\mathbf{i}} \right) \otimes 1) =
\sum_{\mathbf{i}} \alpha_\mathbf{i}\hat{\mu} (x_{(i_1,\ldots,i_r)}) \hat{\mu'}
(x_{(i_{r+1},\ldots,i_{r+s} )}) \lambda \lambda'. \]
On the other hand
\begin{multline*} (\phi(\lambda \otimes \mu) * \phi (\lambda' \otimes \mu') )
(1\otimes  \left(\sum_\mathbf{i}
\alpha_\mathbf{i} x_{\mathbf{i}} \right) \otimes 1) \\=  \phi(\lambda
\otimes \mu) \otimes \phi (\lambda' \otimes \mu') \sum_\mathbf{i} \alpha_{\mathbf{i}}
(1\otimes x_{(i_1,\ldots,i_r)} \otimes 1) \otimes (1 \otimes
x_{(i_{r+1},\ldots,i_{r+s} )} \otimes 1) \\=  \sum_{\mathbf{i}}\alpha_{\mathbf{i}} \hat{\mu} (x_{(i_1,\ldots,i_r)}) \hat{\mu'}
(x_{(i_{r+1},\ldots,i_{r+s} )}) \lambda \lambda'.
 \end{multline*}
This completes the proof.
\end{proof}

\subsection{Quadratic Poisson enveloping algebras are Koszul}

A polynomial Poisson algebra $A =k[x_1,\ldots,x_n]$ is a left Poisson
module over itself in the obvious way.  $A$ is therefore a
$P(A)$-module, isomorphic to the quotient of $P(A)$ by the left ideal
generated by the $\Omega(y_i)$.

\begin{lem} \label{quadratic_poisson} Let $A=k[x_1,\ldots,x_n]$ be a
  polynomial algebra, graded with each $x_i$ in degree one, equipped
  with a Poisson bracket $\{-,-\}$ such that $\{x_i,x_j\} \in A_2$ for
  all $i$ and $j$. Then $P(A)$ is a Koszul algebra and $A$ is a Koszul
  $P(A)$-module. \end{lem}

A Poisson algebra or bracket with this property will be called quadratic.

\begin{proof}
  If we place all $y_i$ and $\Omega(y_i)$ in degree one, the defining
  relations (\ref{rel1}), (\ref{rel2}) and (\ref{rel3}) of $P(A)$ are
  homogeneous of degree two.  Applying Corollary \ref{poisson_pbw}
  shows that $P(A)$ is a quadratic algebra with a PBW basis, and such
  algebras are Koszul by a result of Priddy \cite[Theorem 4.3.1]{PP}.

  When $\Omega(A)$ is a projective $A$-module, which holds when $A$ is
  polynomial, Huebschmann \cite[p.66]{Hueb} points out that
  $P(A)\otimes_A \Lambda^*_A (\Omega(A))$ with the
  Cartan-Chevalley-Eilenberg differential is a projective resolution
  of $A$ over $P(A)$.  Since this is clearly a linear resolution, $A$
  is a Koszul $P(A)$-module by \cite[p.20]{PP}.
\end{proof}

\begin{cor} \label{poisson_quad_dual}$P(A)^!$ is generated by $y_i^*$ and $\Omega(y_i^*)$
  subject to
\begin{gather*}
\Omega(y_i)^* \Omega(y_j)^* + \Omega(y_j)^* \Omega(y_i)^* \,\,\,\,\,\, 1
\leq i \leq j \leq n \\
\Omega(y_m)^* y_n^* + y_n^* \Omega(y_m)^* + \sum_{i<j}
\frac{\partial\{x_i,x_j\} }{\partial x_m \partial x_n} \Omega(y_i)^*
\Omega(y_j)^* \\
y_m^* y_n^* + y_n^* y_m^* - \sum_{i<j} \frac{\partial\{
  x_i,x_j\}}{\partial x_m \partial x_n} (\Omega(y_i)^* y_j^* -
\Omega(y_j)^* y_i^* ).
\end{gather*}
The subalgebra $O=\langle\Omega(y_1)^*,\ldots,\Omega(y_n)^*\rangle$ of
$P(A)^!$ is exterior of rank $n$.
\end{cor}

\begin{proof}
The given presentation for $P(A)^!$ follows from the presentation of $P(A)$
given by relations (\ref{rel1}), (\ref{rel2}) and (\ref{rel3}).

\cite[Corollary 2.2]{PP} says that if $M$ is a Koszul module over a
Koszul algebra $\Lambda$ and $M^!$, $\Lambda^!$ are their Koszul duals
then $h_\Lambda(t)h_{\Lambda^!}(-t)=1$ and
$h_{M^!}(t)=h_{\Lambda^!}(t)h_M(-t)$ where $h_M(t), h_\Lambda(t), h_{M^!}(t), h_{\Lambda^!}$ are the Hilbert
series of $M.\Lambda,M^!,\Lambda^!$.  Applying this with $M=A$ whose
Hilbert series is $(1-t)^{-n}$ and $\Lambda=P(A)$ whose Hilbert series
is $(1-t)^{-2n}$ shows that $h_{P(A)^!}=(1+t)^{2n}$ and the Koszul dual
of the $P(A)$-module $A$ has Hilbert series $(1+t)^n$.

By definition, $A^! = P(A)^! /
P(A)^!(y_1^*,\ldots,y_n^*)$ and it follows by repeatedly using the
second relation that this quotient is spanned by $O +
P(A)^!(y_1^*,\ldots,y_n^*)$.  But $\dim A^!=2^n$, so $\dim O \geq 2^n$.
Since $O$ is certainly a quotient of an exterior algebra of rank $n$ we
in fact have $\dim O = 2^n$. The last statement follows.
\end{proof}

At the moment $A^!$ could mean two different things: the Koszul dual of
the algebra $A$ and the quadratic dual of the $P(A)$-module $A$.
In what follows $A^!$ always refers to the algebra $O$
with $P(A)$-module structure induced by the vector space isomorphism
between $O$ and $P(A)^!/P(A)^!(y_1^*,\ldots,y_n^*)$ from the proof of
the previous corollary.

\begin{rem}
  The quadratic dual of $P(A)$ can be described as
  follows.  A Poisson superalgebra (or graded Poisson algebra) $B$ is a graded-commutative
  algebra, that is $a b = (-1)^{|a||b|}b a$ for $a,b$ homogeneous,
  equipped with a bilinear bracket $\{-,-\}$ such that
\begin{gather*} \{a,b\}=(-1)^{1+|a||b|}\{b,a\} \\
 (-1)^{|a||c|}\{a,\{b,c\}\} + (-1)^{|a||b|}\{b, \{c,a\}\} +
(-1)^{|b||c|}\{c,\{a,b\}\}=0 \\
 \{a,bc\} = \{a,b\}c+ (-1)^{|a||b|}b\{a,c\}
\end{gather*}
 for all homogeneous $a,b,c \in B$.

A quadratic Poisson bracket on the polynomial algebra $A$ is
determined by a map $b : \Lambda^2 (A_1) \to S^2(A_1)$, where
$\Lambda^2$ and $S^2$ are the exterior and symmetric squares and $A_1$
is the vector space spanned by the $x_i$.  The dual
map $b^* : S^2(A_1^*) \to \Lambda ^2 (A_1^*)$ allows us to define a
quadratic Poisson superalgebra structure on the exterior algebra
generated by the $x_i^*$ by
\[ \{ x_i^*, x_j ^*\} := b^*(x_i^*, x_j^*). \]
The quadratic dual of $P(A)$ is
isomorphic to the Poisson enveloping algebra of the exterior algebra
generated by the $x_i^*$ with graded Poisson bracket determined by $b^*$.
\end{rem}

\begin{cor} \label{two_poisson_dgas}
Let $A$ be a polynomial algebra with a quadratic Poisson bracket.
Then the Koszul cocomplex $\bar{K}_{P(A)}(A) = A \otimes _k A^!$ with its natural multiplication and
differential $e_{P(A)}$ is a differential graded algebra whose
cohomology is isomorphic as an algebra to $\HP^*(A)$.
\end{cor}

\begin{proof}
We will show that $\bar{K}_{P(A)}(A)$ with differential $e_{P(A)}$ is
isomorphic as a differential graded algebra to
$\Alt^*_A(\Omega(A),A)$.  As in Subsection
\ref{subsec:poisson_cohomology} we identify this with
$ A \otimes \Lambda^*(V^*) $
where $V^*$ is the span of the $\Omega(x_i)^*$, and map
$ A \otimes_k \Lambda^*(V^*)\to \bar{K}_{P(A)}(A) $ by $\phi: a \otimes
\Omega(x_i)^* \mapsto a \otimes \Omega(y_i)^*$.  This is clearly an
isomorphism of algebras; we only have to show that it respects the
differential.  Since $A \otimes_k \Lambda^*(V^*)$ is generated as an
algebra by the $x_i \otimes 1$ and $1 \otimes \Omega(x_i)^*$ it is
enough to check that $\phi(\delta(x_i \otimes 1) )= e_{P(A)} (x_1 \otimes 1)$
and $\phi (\delta(1 \otimes \Omega(x_i)^*))= e_{P(A)}( 1 \otimes
\Omega(y_i)^*)$, where $\delta$ is the differential corresponding to the
Cartan-Chevalley-Eilenberg differential $d$ on
$\Alt^*_A(\Omega(A),A)$.  The element $x_i \otimes 1 \in A \otimes_k
\Lambda^*(V^*)$ corresponds to $f_i : 1 \otimes 1 \mapsto x_i \in
\Alt^*_A(\Omega(A),A)$, and
\[ d(f_i) (1\otimes \Omega(x_j)) = \Omega(x_j) f(1\otimes 1) =
\{x_j,x_i\} \]
so that $\delta( x_i \otimes 1) = \sum_j \{x_j,x_i\}\otimes
\Omega(x_j)^*$, and the image of this under $\phi$ is $\sum_j
\{x_j,x_i\}\otimes \Omega(y_j)^*$.  Applying 
\[ e_{P(A)} = \sum_{j} \left( y_j \otimes y_j^* + \Omega(y_j)\otimes
\Omega(y_j^*)\right) \]
to $x_i \otimes 1$ gives $\sum_j \Omega(y_j) \cdot x_i \otimes
\Omega(y_j)^*= \sum_j \{x_j,x_i\} \otimes \Omega(y_j)^*$, so
$\phi(\delta(x_i \otimes 1) )= e_{P(A)} (x_1 \otimes 1)$ holds.

$1 \otimes \Omega(x_i)^*$ corresponds to the element $f_i$ of
$\Alt_A(\Omega(A),A)$ that sends $1\otimes \Omega(x_j)$ to $1$ if
$i=j$ and $0$ otherwise. Therefore
\[ df_i(1 \otimes \Omega(x_j)\wedge \Omega(x_k) ) =-df_i(1\otimes_A
\Omega(\{x_j,x_k\}) \]
which is equal to minus the coefficient of $\Omega(x_i)$ in
$\Omega(\{x_j,x_k\})$.  Therefore \[\delta (1 \otimes \Omega(x_i)^*) = -
  \sum_{j<k} c_{ijk} \otimes \Omega(x_j)^*\wedge\Omega(x_k)^*\] where
$c_{ijk}$ is the coefficient of $\Omega(x_i)$ in
$\Omega(\{x_j,x_k\})$.  On the other hand, applying $e_{P(A)}$ to $1
\otimes \Omega(y_i)^*$ gives 
\[\sum_{r} y_r \otimes y_r^* \Omega(y_i)^* = -\sum_{j<k} \left(\sum_r
  y_r \frac{\partial\{x_j,x_k\}}{\partial x_r\partial x_i} \right)
\otimes \Omega(y_j)^*\Omega(y_k)^* \]
using the second relation from Corollary \ref{poisson_quad_dual},
which completes the proof.
\end{proof}

The polynomial algebra $A=k[x_1,\ldots,x_n]$ admits derivations
$\partial_i=\frac{\partial}{\partial x_i}$ defined in the usual way.
Similarly its quadratic dual $A^!=k\langle \Omega(y_1)^*,\cdots,\Omega(y_n)^*
\rangle$ has graded derivations $\partial_i^*$ defined by 
\begin{equation*}
	\partial_i^*( \Omega(\mathbf{y})^{*\mathbf{b}})=(-1)^{\sum_{j<i} b_j}b_i \Omega(\mathbf{y})^{*\mathbf{b}-e_i}
\end{equation*}
where $e_i$ is the vector with a $1$ in position $i$ and zeroes 
elsewhere and $\Omega(\mathbf{y})^{*\mathbf{b}}$ denotes 
$\Omega(y_1)^{*b_1}\cdots \Omega(y_n)^{*b_n}$.
We can use these to build a derivation on the Koszul cocomplex:

\begin{lem}
	\label{almost_homotopy}
Let $A=k[x_1,\ldots,x_n]$ be a quadratic Poisson algebra and
let $A \otimes_k A^!$ be the Poisson cocomplex computing $\HP^*(A)$ with
differential $e=e_{P(A)}$.  Let $h = \sum_i \partial_i \otimes
\partial^*_i$.  Then \begin{equation}
		he + eh = \sum_i \partial_i H_i \otimes 1 + 1
		\otimes \partial_i^*  H^*_i
		\label{ht}
	\end{equation}
	where $H_i(x)=\{x_i,x\}$ and
	$H_i^*(X)=\{\Omega(y_i)^*,X\}=y_i^*\cdot X$.
\end{lem}
\begin{proof}
Since $h$ and $e$ are (graded) derivations, so is $he+eh$.  We first show
that the right hand side of (\ref{ht}) is a derivation.

Since $\{-,-\}$ and its dual are (graded) derivations when one argument
is fixed,
\begin{align} \label{sy1} \{x,y\}&=\sum_i \partial_i(x) H_i(y)
	= -\sum_i \partial_i(y)H_i(x) \\
\nonumber\{z,w\}&=(-1)^{|z|+1}\sum_i \partial^*_i(z) H_i^*(w) =
	(-1)^{|z||w|+|w|}\sum_i \partial^*_i(w) H_i^*(z) \end{align}
for $z,w \in A^!$ and $x,y \in A$. Furthermore 
\begin{equation*}  \partial_i H_i(xy) =  \partial_i(x)H_i(y)+x\partial_iH_i(y)+\partial_iH_i(x) y +
	H_i(x)\partial_i(y)\end{equation*}
so summing over $i$ and applying (\ref{sy1}) shows $\sum_i
\partial_i H_i$ is a derivation.  Similarly so is $\sum_i
\partial_i^*H_i^*$, and thus
$\sum_i \partial_i H_i \otimes 1 + 1 \otimes
\partial_i^*H_i^*$ is a derivation.

It is easy to see that $he+eh$ and $\sum_i \partial_iH_i \otimes 1 + 1
\otimes \partial_i^*H_i^*$ agree on $x_r\otimes 1$ and $1\otimes
\Omega(y_r)^*$ for any $r$.  Since these elements generate $A\otimes
A^!$, the two derivations are equal. \end{proof}

The utility of this result
is that in some cases, including that of the semiclassical limits we are
interested in, $he+eh$ sends each element of the PBW basis 
to a scalar multiple of itself.  
Therefore some multiple of $h$ will be a contracting homotopy for
certain parts of the Koszul cocomplex.

\subsection{$q$-deformations of the Koszul cocomplex}

Let $\sA$ be a Koszul $k(q)$-algebra which has a semiclassical limit
$A$ such that $\sA^e$ is a $q$-deformation of $P(A)$
Then $\HH^*(\sA)$ is computed by the differential
graded algebra $\sA\otimes
\sA^!$ with the differential $e_{\sA^e}$ and $\HP^*(A)$ is computed by
$A \otimes A^!$ with differential $e_{P(A)}$.  We want to show that
the first of these DGAs is a $q$-deformation of the second.

\begin{prop} \label{defm_of_Kos_cocomp} 
Let $\sA$ be a Koszul $k(q)$-algebra minimally generated by
$\sx_1,\ldots, \sx_n$ and let $\cA$ be the $\kq$-form of $\sA$
generated by the $\sx_i$ and admitting
a semiclassical limit $A$ which is polynomial on the images of the $\sx_i$.
Suppose the map $P(A) \twoheadrightarrow \cA'/(q-1)\cA'$  of Lemma
\ref{surject} is an isomorphism.  Then the differential graded algebra
$\bar{K}_{\sA^e}(\sA)$ is a $q$-deformation of $\bar{K}_{P(A)}(A)$.
\end{prop}

\begin{proof} We split the proof into sections.

\mbox{}

\noindent \textbf{Hilbert series}. Under our hypotheses,
\[ \dim A_r = \operatorname{rank} \cA_r = \dim \sA_r \]
for any $r$.  The first equality holds because $\cA_r$ is free as a
$\kq$-module, being finitely-generated torsion-free.  For the
second, if $b_1,\ldots$ is a basis of $\sA_r$ then by
multiplying by an appropriate scalar we may assume the $b_i$ lie in
$\cA_r$. Thus $\operatorname{rank} \cA_r \geq \dim \sA_r$.  But any
$\kq$-linearly independent set in $\cA_r$ is $k(q)$-linearly
independent in $\sA_r$, so the opposite inequality holds.  We get that
$\sA$ has the same Hilbert series as the polynomial algebra $A$.  The Hilbert series of a Koszul algebra determines that of its Koszul
dual by \cite[Corollary 2.2.2]{PP}, so $\sA^!$ has the same Hilbert
series as the exterior algebra $A^!$.

\mbox{}

\noindent \textbf{Relations in} $\sA^!$.
  Write $\sA = T(V)/(R)$ as a quadratic algebra where $V$ has a basis
  $v_1,\ldots,v_n$ whose images in $\sA$ are
  $\sx_1,\ldots,\sx_n$. Existence of the semiclassical limit implies
  that for
  each $i<j$ there is an element $r_{ij}$ of $R$ of the form
  \[ v_i \otimes v_j - v_j \otimes v_i - (q-1) \hat\beta_{ij} \] where
  $\hat\beta_{ij}$ lies in the $\kq$-span of the $v_s\otimes v_t$.
By substituting $r_{st}$ into the $\hat{\beta_{ij}}$ we obtain
elements $r'_{ij}$ of $R$ of the form
\[ v_i \otimes v_j - v_j \otimes v_i - (q-1) \hat\beta'_{ij}+(q-1)^2 
	\gamma_{ij} \] 
where only tensors of the form $v_s \otimes v_t$ for  $s<t$ appear in
$\beta'_{ij}$, and the image of $\hat\beta'_{ij}$ in $A$ is
$\{x_i,x_k\}$.  These $r'_{ij}$ are $k(q)$-linearly independent since
they are even linearly independent modulo $q-1$.  They form a $k(q)$-basis
of $R$, for if $R$ had larger dimension $\dim \sA_2$ would be smaller than
$n(n+1)/2$.

Let $v_1^*,\ldots,v_n^*$ be the basis of $V^*$ dual to $v_1,\ldots,v_n$.
For $r<s$ define
\[ \rho_{rs} = v_r^*\otimes v_s^* + v_s^* \otimes v_r^* -
	(q-1)\sum_{k<l}\frac{\{x_k,x_l\}}{\partial x_r \partial
	x_s}v_l^*\otimes v_k^* \]
where $x_i$ is the image of $\sx_i$ in $A$. Furthermore define
\[ \rho_{rr} = v_r^*\otimes v_r^* - (q-1)\sum_{k<l}
	\frac{1}{2}\frac{\{x_k,x_l\}}{\partial x_r^2} v_l^*\otimes v_k^*
	\]
Then for any $r\leq s$ and any $i<j$ we have $\rho_{rs}(r_{ij}) \in
(q-1)^2 \kq$.  As before, the $\rho_{rs}$ are linearly independent and
there are some elements $\varepsilon_{rs}$ in the $\kq$-span of the $v_i^* \otimes
v_j^*$ such that 
\begin{equation} \rho_{rs} + (q-1)^2 \varepsilon_{rs} \,\,\,\,\,\,\, r\leq
	s \label{rels_in_Ashriek} \end{equation}
is a basis of $R^\perp$.  Recall that $\sA^! = T(V^*)/(R^\perp)$ and write
$\sx_i^*$ for the image of $v_i^*$ in $\sA^!$. 
Let $\cB$ be the $\kq$-subalgebra of $\sA^!$ generated by the elements
$(1-q)\sx_i^*$.  The relations
(\ref{rels_in_Ashriek}) show that 
\[ \sx_i^* \sx_j^* + \sx_j^*\sx_i^* =0 \mod (q-1) \]
Since $\sA^!$ and hence $\cB$ has the same Hilbert series as $A^!$, $\cB/(q-1)\cB$ is isomorphic
to the exterior algebra $A^!$.  

Let $\cB'$ be the subalgebra of $\sA^! \otimes_{k(q)} (\sA^!)^{\opp}$
generated by $\sx_i ^* \otimes 1 + 1 \otimes \sx_i^*$ and
$(1-q)\otimes \sx_i^*$ for $1 \leq i \leq n$.  

\mbox{}

\noindent \textbf{$\cA$ is a $\cA'$-module, $\cB$ is a $\cB'$-module}.
$\sA$ is a $\sA^e$-module, and we claim that the action of
$\cA'\subseteq \sA^e$ preserves $\cA$.  Certainly elements
$\sx_i\otimes 1$ preserve $\cA$, and $(\sx_i \otimes 1 - 1\otimes
  \sx_i) \cdot \cA \subseteq (q-1)\cA$ because of the existence of the
  semiclassical limit.  

Recall that the action of $\sA^! \otimes _{k(q)}(\sA^!)^{\opp}$ on
$\sA^!$ is \[(\lambda \otimes \mu') \cdot x = (-1)^{|\mu||x| +
  |\mu|(|\mu|+1)/2} \lambda x \mu.\]  Clearly the action of
$(1-q)\otimes \sx_i^*$ preserves $\cB$.  The relations
(\ref{rels_in_Ashriek}) show that the  elements $\sx_i ^* \otimes 1 + 1 \otimes
\sx_i^*$ send generators of $\cB$ to $\cB$; the
general result follows because $\sx_i ^* \otimes 1 + 1 \otimes
\sx_i^*$ acts by graded derivations on $\cB$.

Let $\cC=\cA \otimes_\kq \cB$, and $e_{\cC}:\cC \to \cC$ be the map induced by left-action of
\[ \sum_i \left( (\sx_i \otimes 1) \otimes (\sx_i^* \otimes 1 + 1
  \otimes \sx_i^*) + \frac{\sx_i \otimes 1 - 1 \otimes
    \sx_i}{q-1}\otimes ((1-q)\otimes \sx_i^* ) \right) \in \cA'
\otimes \cB'\] This makes it clear that $\cC$ with differential $e_{\cC}$ is
a DGA which is a $\kq$-form of $\bar{K}_{\sA^e}(\sA)$.  Therefore all
we need to complete the proof is:

\mbox{}

\noindent \textbf{$\mathcal{C}/(q-1)\mathcal{C}$ is isomorphic as a
  DGA to  $\bar{K}_{P(A)}(A)$}.  We consider $\bar{K}_{P(A)}(A)$ as the
complex $A \otimes_k A^!$ as in the proof of
Corollary \ref{two_poisson_dgas}.  

$\bar{K}_{P(A)}(A)$ is isomorphic as an algebra to $\cC/(q-1) \cC
\cong (\cA / (q-1)\cA)\otimes (\cB/(q-1)\cB)$.
Write $\phi:
\bar{K}_{P(A)}(A) \to (\cA / (q-1)\cA)\otimes (\cB/(q-1)\cB)$ for the
algebra isomorphism that sends
$x_i \otimes 1$ to the image of $\sx_i \otimes 1$ and $1 \otimes
\Omega(y_i)^*$ to the image of $(1-q)\otimes \sx_i^*$ in
$\cC/(q-1)\cC$. We claim $\phi$ respects the
differentials on these DGAs, and since $\bar{K}_{P(A)}(A)$ is generated by the $x_i\otimes 1$ and $1 \otimes \Omega(y_i)^*$ it is enough
to check these elements.

\[ e_{P(A)} (x_i \otimes 1)=\sum_j \{x_j,x_i\} \otimes
\Omega(y_j)^* \]
\[ e_{\cC}(\sx_i \otimes 1)= \sum_j \beta(\sx_j,\sx_i) \otimes (1-q)
\sx_j^* \]
so $\phi(e_{P(A)} (x_i \otimes 1))$ equals $e_{\cC}(\phi (x_i \otimes
1)) + (q-1)\cC$.

\[e_{P(A)} (1 \otimes \Omega(y_i)^*) = - \sum_{j<k} \sum_r 
\frac{\partial\{x_j,x_k\}}{\partial x_i \partial x_r} x_r\otimes
\Omega(y_j)^* \Omega(y_k)^* \]
\begin{gather*} e_{\cC}( 1 \otimes (1-q) \sx_i^*) \equiv \sum_r \sx_r \otimes
  (1-q)(\sx_r^* \sx_i^* + \sx_i^* \sx_r^* )\\ = \sum_r \sum_{j<k}
  \frac{\partial\{x_j,x_k\}}{\partial x_i \partial x_r} \sx_r^* \otimes
  (1-q)^2 \sx_k^* \sx_j^* \end{gather*}
modulo $(q-1)\cC$.
\end{proof}

If $\alpha \in \kq$ and $M$ is a $\kq$-module we say $M$ has
$\alpha$-torsion if there is a non-zero element of $M$ annihilated by $\alpha$.

\begin{cor} \label{HH_eq_HP}
  If $H^*(\mathcal{C})$
  has no $(q-1)$-torsion then 
 $\HH^*(\sA)$  is a $q$-deformation of $\HP^*(A)$.
\end{cor}

\begin{proof}
  $\cC \otimes _\kq k\cong \bar{K}_{P(A)}(A)$ so $H^*(\cC \otimes_\kq
  k)\cong \HP^*(A)$, and $\cC \otimes _\kq k(q)\cong
  \bar{K}_{\sA^e}(\sA)$ so $\HH^*(\sA)\cong H^*(\cC \otimes_\kq
  k(q))\cong H^*(\cC)\otimes_\kq k(q)$.

We first show that $\HP^r(A)$ and $\HH^r(\sA)$ have the same Hilbert
series for any $r$.  The universal coefficient theorem gives an exact
sequence of graded modules
\begin{equation}\label{univ_coeff} 
0 \to H^{rj}(\mathcal{C}) \otimes_{k[q^{ \pm 1 }]} k
 \to H^{rj}(\mathcal{C}\otimes_{k[q^{ \pm 1 }]} k) 
\to \Tor_1 ^{k[q^{ \pm 1 }]}(H^{r-1,j}(\mathcal{C}),k) \to 0
\end{equation}
for any $j$.
Because $\kq$ is a principal ideal domain,
\[ \Tor _1 ^{k[q^{ \pm 1 }]} \left(\frac{k[q^{ \pm 1 }]}{(f)},k\right) = \ker (k
\stackrel{f \cdot }\to k) \] and the Tor group in (\ref{univ_coeff})
vanishes under our hypothesis.  So for any $r,j$, $\HP^{rj}(A) \cong
H^{rj}(\cC)\otimes_\kq k$ which has the same dimension as $H^{rj}(\cC
\otimes _\kq k(q)) \cong \HH^{rj}(\sA)$.

Consider the $\kq$-form $H^*(\cC) \otimes _\kq 1$ of
$H^*(\cC)\otimes_\kq k(q)$.  This is isomorphic as a $\kq$-algebra to the
quotient of $H^*(\cC)$ by its torsion ideal $T$ (the ideal of all
elements annihilated by some $\alpha \in \kq\setminus\{0\}$).  Under our
hypothesis $-\otimes_\kq k$ kills all torsion summands of $H^*(\cC)$,
so $H^*(\cC)\otimes_\kq k \cong (H^*(\cC)/T) \otimes_\kq k$.  Thus
$\HH^*(\sA) \cong H^*(\cC)\otimes_\kq k(q)$ is a $q$-deformation of
$\HP^*(A) \cong H^*(\cC)\otimes_\kq k$ via the $\kq$-form
$H^*(\cC)\otimes 1 \cong H^*(\cC)/T$.
\end{proof}

In practise the hypothesis of Corollary \ref{HH_eq_HP} can be checked
using the following lemma.

\begin{cor}  \label{lifting}
If every
  element of $\ker ( e_{\cC}\otimes_{\kq} k)$ lifts to an element of $\ker
  e_{\cC}$, then $H^*(\mathcal{C})$ has no $(q-1)$-torsion. \end{cor}

\begin{proof}
  It is enough to show that $\mathcal{C}/\im e_{\cC}$ has no
  $(q-1)$-torsion under these hypotheses.  Since $\kq$ is a principal
  ideal domain, maps between free $\kq$-modules can be written in Smith
  Normal Form. It follows easily that $\mathcal{C}/\im e_{\cC}$ has
  $(q-1)$-torsion if and only if $\rank (e_{\cC} \otimes k) < \rank
  e_{\cC}$ on some graded piece, if and only if the kernel of $e_{\cC}
  \otimes k$ has rank larger than that of $e_{\cC}$ on that graded
  piece.  If every element of $\ker(e_{\cC}\otimes k)$ lifts, this
  cannot happen.
\end{proof}

\section{The quantum plane} \label{sec:quantum_plane}

In this section we illustrate some of the results of the previous sections
in the case of the quantum plane.  The Hochschild cohomology of
quantum affine space was described in \cite[\S3.3]{Sitarz} (though
it appears difficult to obtain an explicit expression for arbitrary
$n$), and it is known (\cite[Corollary 3.5.2]{LR}, \cite{Monnier})
that it agrees with the Poisson cohomology of the semiclassical limit.

The coordinate ring $\sA$ of the quantum plane as
defined in Section \ref{quant_aff_sp_ex} is Koszul since it is a
quadratic algebra with a PBW basis of polynomial type.  Therefore by
Theorem \ref{poly_pbw} and Proposition \ref{defm_of_Kos_cocomp},
$\bar{K}_{\sA^e}(\sA)$ is a $q$-deformation of $\bar{K}_{P(A)}(A)$,
where $A$ is the semiclassical limit of $\sA$.

The Koszul dual of $\sA$ is the quantum
exterior algebra $\sA^!$,
generated by $\sx^*,\sy^*$ subject to $\sx^{*2}=\sy^{*2}=0$ and
$ q\sx^* \sy ^* + \sy^* \sx^* = 0$ .

Using this and (\ref{diff}) we can compute the Koszul cocomplex $\bar{K}_{\sA^e}(\sA)$.
For $m \in \mathbb{Z}$ write $[m]$ for the $q$-integer $\frac{q^m-1}{q-1}$.  The maps in the Koszul cocomplex are
\begin{align*} e_0 ( \sy^b\sx ^a) &= [b]\sx^{a+1}\sy^n \otimes \sx^* - [a] \sx^a
\sy^{b+1}\otimes \sy^* \\
 e_1 ( \sy^b\sx^a\otimes \sx^*)& =[a-1]\sx^a \sy^{b+1}\otimes \sx^*
 \sy^*\\
 e_1 ( \sy^b\sx^a\otimes \sy^*)&=[b-1]\sx^{a+1 }\sy^{b} \otimes\sx^*
 \sy^* \end{align*}
 so, noting that $\gcd([a],[b])=[\gcd(a,b)]$ in $\kq$, 
\[ \ker e_0 = \kq\cdot 1,\,\,\, (\ker e_1)_{n+1}= \kq \left \langle \frac{ [b]
  \sy^b \sx^{n-b}\otimes \sx^* - [n-b-1]\sy^{b+1} \sx^{n-1-b}\otimes\sy^* }{[\gcd(b,n-b-1)]}
\right \rangle \]
Therefore $H^0(\cC)= \kq \cdot 1$,
\begin{align*} H^1(\cC)& \cong \kq \sx \otimes \sx^* \oplus \kq \sy
  \otimes\sy^* \oplus T_1 \\
H^2(\cC)&\cong \kq \otimes \sx^*\sy^* \oplus \kq \sy\sx\otimes\sx^*\sy^*
\oplus T_2
\end{align*}

where $T_1$ and $T_2$ are direct sums of $\kq$-modules of the form
$\kq/([a])$ for $a\in \mathbb{Z}$.  Therefore there is no $(q-1)$-torsion in $H^*(\cC)$, and 
$\HH(\sA)$ is a $q$-deformation of $\HP(A)$. Furthermore $\HH(\sA)$ is
one-dimensional in homological degree zero and two-dimensional in
degrees $1$ and $2$.

In fact, for any quantum affine space two the Koszul cocomplexes are actually isomorphic after
a change of base field: 
 $\bar{K}_{\sA^e}(\sA) \cong
  \bar{K}_{P(A)}(A)\otimes _k k(q)$ as
  complexes.

  \section{$2\times 2$ quantum matrices}\label{sec:matrices}
In this section we will compute the Poisson cohomology of $M$, showing
that it agrees with the Hochschild cohomology of $\sM$.
 In describing elements
of $\bar{K}_{P(M)}(M)$ we write $\Omega^{ij}_{kl}$ for
$\Omega(a)^{*i}\Omega(b)^{*j}\Omega(c)^{*k}\Omega(d)^{*l}$ and we omit
the $\otimes$ symbol. Similarly in $\bar{K}_{\sM^e}(\sM)$ we write $\sff^{ij}_{kl}$ for
$\sd^l\scc^k\bb^j\sa^i$ and
$\tilde\Omega^{ij}_{kl}$ for
$\tilde\Omega(\sa)^{*i}\tilde\Omega(\bb)^{*j}\tilde\Omega(\scc)^{*k}\tilde\Omega(\sd)^{*l}$
and omit the $\otimes$ symbol.

We begin by summarizing the results. 

\begin{theo}\label{HPM}
  $\HH^*(\sM)$ is a $q$-deformation of $\HP^*(M)$.  The Poisson centre
  of $M$ is polynomial in $\Delta=ad-bc$, and for $i \leq 3$, $\HP^i(M)$
  is a free $k[\Delta]$-module freely generated by the images of the
  following elements:
\begin{center}
\begin{tabularx}{\textwidth}{l|X}
 $ i$ & Representatives of generators of $\HP^i(A)$\\
\hline
 0 & $1$ \\
 1 & $a \Omega^{10}_{00}+d\Omega^{00}_{01},
 a\Omega^{10}_{00}+c\Omega^{00}_{10},b\Omega^{01}_{00}+d\Omega^{00}_{01} $ \\
2 & 
$b^j \Omega^{10}_{01},
c^k \Omega(a)^{10}_{01}$ for $j,k \neq 2$, 
$bc \Omega^{01}_{10}
- ab \Omega^{11}_{00}+ac\Omega^{10}_{10}$,
$b(a\Omega^{11}_{00} + d\Omega^{01}_{01})$, 
$c(a\Omega^{11}_{00} + d\Omega^{01}_{01})$,
$b(a\Omega^{10}_{10} + d\Omega^{00}_{11})$,
$c(a\Omega^{10}_{10} + d\Omega^{00}_{11}) $ \\
3 & $b^j\Omega^{11}_{01},
c^k\Omega^{11}_{01},
b^k\Omega^{10}_{11},
c^j\Omega^{10}_{11}$
for $j \neq 3$ and
$b^jc^k(a\Omega^{11}_{10} -d\Omega^{01}_{11})$ for $j+k=2$.
\end{tabularx} \end{center}
As a $k[\Delta]$-module, $\HP^4(M)$ is the direct sum of a trivial module
generated by $\Omega^{11}_{11}$ and free summands generated by
$bc\Omega^{11}_{11}$, $b^j
\Omega^{11}_{11}$ and $c^k\Omega^{11}_{11}$ 
 for $j,k>0$.
\end{theo}

The theorem is proved as follows.  Firstly,
$\sM$ has a PBW basis of polynomial type, so we may apply
Theorem \ref{poly_pbw} and Proposition \ref{defm_of_Kos_cocomp} to
get that $\bar{K}_{\sM^e}(\sM)$ is a $q$-deformation of
$\bar{K}_{P(M)}(M)$. 
We compute $\HP^*(M)$ directly using the Koszul
cocomplex and show that each cocycle lifts to one for
$\bar{K}_{\sM^e}(\sM)$, so that Corollary \ref{lifting} applies and
$\HH^*(\sM)$ is a $q$-deformation of $\HP^*(M)$.  This is simplified by
two observations: first, the boundary of
$d^lc^kb^ja^i\Omega^{i'j'}_{k'l'}$ lifts to that of
$\sd^l\scc^k\bb^j\sa^i\tilde\Omega^{i'j'}_{k'l'}$ so we only need to
lift non-bounding cocycles, and second $\Delta=ad-bc$ lifts to
$\Delta_q=ad-qbc$ so we only need lift a $k[\Delta]$-generating set for
the non-bounding cocycles. In practise the lifting
is always the most straightforward one possible:
$a\Omega^{10}_{00}+b\Omega^{01}_{00}$ lifts to $\sa
\tilde\Omega^{10}_{00}+\bb\tilde\Omega^{01}_{00}$ and so on.

We need to know the differentials on both Koszul cocomplexes explicitly
to check the lifting condition.  Since
the differential on
$\bar{K}_{P(M)}(M)$ is the `reduction mod $q-1$' of that for
$\bar{K}_{\sM^e}(\sM)$, we give the latter in the tables that follow.

\begin{multline*} e_0(\sff^{ij}_{kl})= ([j+k]\sff^{i+1,j}_{kl}
+
q^{-1}[2l] \sff^{i,j+1}_{k+1,l-1}
) \tilde\Omega^{10}_{00} 
+([l]-[i])( \sff^{i,j+1}_{k,l}
\tilde\Omega^{01}_{00} +
\sff^{ij}_{k+1,l}
\tilde\Omega^{00}_{10} ) 
\\- ([j+k]
\sff^{ij}_{k,l+1} +
q^{-1}[2i]
\sff^{i-1,j+1}_{k+1,l}
) \tilde\Omega^{00}_{01}\end{multline*}

\begin{center}
\begin{tabularx}{\textwidth}{l|X}
  $x$ & $e_1(\sff^{ij}_{kl} \Omega(x)^*) $ \\
\hline
$\sa$ & $ q^{-1}([i]-[l+1])( \sff_{kl}^{i,j+1} \tilde\Omega^{11}_{00}
+ \sff_{k+1,l}^{ij} \tilde\Omega^{10}_{10})+ ([j+k]\sff_{k,l+1}^{ij}  +
q^{-1}[2i]\sff_{k+1,l}^{i-1,j+1} )\tilde\Omega^{10}_{01} $ \\
$\bb$ & $([j+k-1]\sff_{kl}^{i+1,j}+q^{-2}
[2l]\sff_{k+1,l-1}^{i,j+1} )
\tilde\Omega^{11}_{00} 
+ ([i]-[l])\sff_{k+1,l}^{ij}
\tilde\Omega^{01}_{10}
- q^{i-1}[2]\sff_{k+1,l}^{ij}
\tilde\Omega^{10}_{01}
+ ( [j+k-1]\sff_{k,l+1}^{ij} +
q^{-2}[2i]\sff_{k+1,l}^{i-1,j+1} )
\tilde\Omega^{01}_{01}$ \\
$\scc$ & $ ([j+k-1]\sff_{kl}^{i+1,j} +q^{-2}
[2l]\sff_{k+1,l-1}^{i,j+1})  \tilde\Omega^{10}_{10}
+ ([l]-[i])\sff_{kl}^{i,j+1}
\tilde\Omega^{01}_{10}
- q^{l-1}[2] \sff_{kl}^{i,j+1}  \tilde\Omega^{10}_{01}
+( [j+k-1] \sff_{k,l+1}^{ij} +q^{-2}[2i]
\sff_{k+1,l}^{i-1,j+1} ) \tilde\Omega^{00}_{11}$ \\
$\sd$ & $ ([l-1]-[i]) (\sff_{kl}^{i,j+1} \tilde\Omega^{01}_{01}
+ \sff_{k+1,l}^{ij}\tilde\Omega^{00}_{11}) 
+ ([j+k]\sff_{kl}^{i+1,j} +
q^{-1}[2l]\sff_{k+1,l-1}^{i,j+1})
\tilde\Omega^{10}_{01}$
\end{tabularx}
\end{center}

\begin{center}
\begin{tabularx}{\textwidth}{l|X}
$x,y $ & $e_2(\sff_{kl}^{ij}\Omega(x)^*\Omega(y)^*) $ \\
\hline
$\sa,\bb$ & $q^{-1}([l+1]-[i]) \sff_{k+1,l}^{ij}
\tilde\Omega^{11}_{10} 
- ( [j+k-1] \sff_{k,l+1}^{ij} +
q^{-2}[2i] \sff_{k+1,l}^{i-1,j+1}) 
\tilde\Omega^{11}_{01}$ \\
$\sa, \scc$ & $
q^{-1} ([i]-[l+1]) \sff_{kl}^{i,j+1}
\tilde\Omega^{11}_{10} 
 - ( [j+k-1]\sff_{k,l+1}^{ij} +
q^{-2}[2i] \sff_{k+1,l}^{i-1,j+1}) 
\tilde\Omega^{10}_{11}$ \\
$\sa, \sd$ & $q^{-1}([i]-[l]) 
( \sff_{kl}^{i,j+1}  \tilde\Omega^{11}_{01} 
+ \sff_{k+1,l}^{ij} \tilde\Omega^{10}_{11} )$ \\
$\bb,\scc$ & $([j+k-2] \sff_{kl}^{i+1,j} + q^{-3}[2l]
\sff_{k+1,l-1}^{i,j+1}) \tilde\Omega^{11}_{10}   
- q^{l-2}[2] \sff_{kl}^{i,j+1} \tilde\Omega^{11}_{01}
+q^{i-2}[2] \sff_{k+1,l}^{ij}
\tilde\Omega^{10}_{11}
-([j+k-2] \sff_{k,l+1}^{ij} + q^{-3}[2i]
\sff_{k+1,l}^{i-1,j+1}) 
\tilde\Omega^{01}_{11}
$ \\
$\bb, \sd$ & $([j+k-1]\sff_{kl}^{i+1,j} + q^{-2}[2l]
\sff_{k+1,l-1}^{i,j+1})
\tilde\Omega^{11}_{01}
+ ([i]-[l-1]) \sff_{k+1,l}^{ij} \tilde\Omega^{01}_{11}
$ \\
$\scc, \sd$ & $([j+k-1] \sff_{kl}^{i+1,j} + q^{-2}[2l]
\sff_{k+1,l-1}^{i,j+1}) 
\tilde\Omega^{10}_{11}
+ ([l-1]-[i]) \sff_{kl}^{i,j+1} \tilde\Omega^{01}_{11}$
\end{tabularx}
\end{center}
\begin{center}
  \begin{tabularx}{\textwidth}{l|X}
    $i',j',k',l'$ & $e_3(\sff_{kl}^{ij}\tilde\Omega^{i'j'}_{k'l'}) $ 
    \\ \hline
    $1,1,1,0$ & $[j+k-2]\sff_{k,l+1}^{ij}\tilde\Omega^{11}_{11} +
    q^{-3}[2i]\sff_{k+1,l}^{i-1,j+1}\tilde\Omega^{11}_{11}$ \\
    $1,1,0,1$ &$ q^{-1}([l]-[i]) \sff_{k+1,l}^{ij} \tilde\Omega^{11}_{11}$\\
    $1,0,1,1$ &$q^{-1}([i]-[l]) \sff_{kl}^{i,j+1}\tilde\Omega^{11}_{11} $ \\
    $0,1,1,1$ & $ [j+k-2] \sff_{kl}^{i+1,j}\tilde\Omega^{11}_{11} + q^{-3}[2l]
    \sff_{k+1,l-1}^{i,j+1}\tilde\Omega^{11}_{11} $   \end{tabularx}
\end{center}

The formulas are obtained by computing directly using the definition of
the Koszul cocomplex and  the presentations
of $\sM$ and $\sM^!$ given in (\ref{m2rels}) and Example \ref{quad_dual_2x2_matrices}.

We want to find a contracting homotopy to show that certain portions of
$\bar{K}_{P(M)}(M)$ are exact.  To this end we define another grading on
$\bar{K}_{P(M)}(M)$:

\begin{defn}
The grade of $a^ib^jc^kd^l
\Omega^{i',j'}_{k',l'}\in \bar{K}_{P(M)}(M)$
is $i-i'-l+l'$.
\end{defn}

The differential $e_{P(M)}$ preserves grade, as is
easily seen from the tables above.

\begin{lem} 
%\begin{equation*}
  $(he+eh)(f^{ij}_{kl}\Omega^{i'j'}_{k'l'})
	 = 2(l-l'-i+i')f^{ij}_{kl}\Omega^{i'j'}_{k'l'}$
       %\end{equation*}
\end{lem}
\begin{proof}
  This is an application of Lemma \ref{almost_homotopy}.  Computing using the bracket formulas
  (\ref{bracket_formulas}) shows
  \[ \left(\partial_aH_a + \partial_bH_b +
    \partial_cH_c + \partial_
    dH_d\right) f
    =2(l-i)f^{ij}_{kl}.\]
  Next, the second relation of Corollary \ref{poisson_quad_dual} gives
  the following description of the action of $P(M)^!$ on $M^!$ (with a
  slight abuse of notation):
  \begin{align*}
    a^* \cdot \Omega(b)^* = b^* \cdot \Omega(a)^* &=
    -\Omega(a)^*\Omega(b)^* \\ 
    a^* \cdot \Omega(c)^* = c^* \cdot \Omega(a)^* &=
    -\Omega(a)^*\Omega(c)^* \\
    a^* \cdot \Omega(d)^* = d^* \cdot \Omega(a)^* &=
   0 \\ 
b^* \cdot \Omega(c)^* = c^* \cdot \Omega(b)^* &=
    -2\Omega(a)^*\Omega(d)^* \\
b^* \cdot \Omega(d)^* = d^* \cdot \Omega(b)^* &=
   0 \\ 
c^* \cdot \Omega(d)^* = d^* \cdot \Omega(c)^* &=
    -\Omega(c)^*\Omega(d)^* 
  \end{align*}
  and that $x^* \cdot \Omega(x)^*=0$ for $x=a,b,c,d$.
  From these, a tedious calculation shows that 
  \[ \left(\partial_a^*H_a^* + \partial_b^*H^*_b +
    \partial_c^*H^*_c + \partial_
    d^*H^*_d\right) \Omega^{i'j'}_{k'l'}
    =2(i'-l')\Omega^{i'j'}_{k'l'}  .
    \]
 Putting these together gives the formula claimed. 
\end{proof}
This means that a suitable scalar multiple of $h$ is a contracting
homotopy on any summand of $\bar{K}_{P(M)}(M)$ with nonzero grade.
Write $K_0(M)$ for the subcomplex of
$\bar{K}_{P(M)}(M)$ consisting of all elements of grade zero, and let $E_n$ be the induced differential in homological
degree $n$. %: so $K_0(M)$ contains elements like $(ad)^ib^jc^k$,
Then $\HP^*(M)\cong H^*(K_0(M), E)$, so we work only with
the cocomplex $K_0(M)$. %We note that $E_*$ is considerably simpler than

\subsection{Computations}
\begin{lem}
  $\HP^0(M)=k[\Delta]$, and $\Delta$ lifts to $\Delta_q = ad-qbc$.
\end{lem}

\begin{proof}
  $K_0(M)_0$ consists of elements of the form $\sum
  \lambda_{ijk}(ad)^ib^jc^k$, and such an element is in $\ker E_0$ if and
  only if it has trivial bracket with $a$, if and only if
\begin{equation*}
	2i\lambda_{i,j-1,k-1}=-(j+k)\lambda_{i-1,j,k}
	\label{recurrence}
\end{equation*}
for all $i,j,k$, where $\lambda_{ijk}$ should be interpreted as zero if any of its
indices are negative.  If $\lambda_{ijk}$ is a solution of this
recurrence and $j>k$ then $\lambda_{ijk}=0$: if $k$ is
minimal such that there exists $j>k$ with $\lambda_{ijk}\neq 0$ then
$(j+k)\lambda_{ijk}=-2(i+1)\lambda_{i+1,j-1,k-1}=0$ contradicting
$\lambda_{ijk}\neq 0$.
By symmetry $\lambda_{ijk}=0$ if $j \neq k$.

Writing $\lambda_{ij}$ for $\lambda_{ijj}$ gives 
\begin{equation*}
	i\lambda_{i,j-1}=-j\lambda_{i-1,j}
	\label{recurrence2}
\end{equation*}
For fixed $i+j$, any solution is a scalar multiple of
$\lambda_{ij}=(-1)^i { i+j \choose i}$, thus $\sum
\lambda_{ij}(ad)^i(bc)^j$ is a polynomial in $\Delta$.  The claim about
lifting is easily verified.
\end{proof}

Therefore $E_0$ is a $k[\Delta]$-map with kernel equal to $k[\Delta]$.
As a $k[\Delta]$-module, $K_0(M)=\langle (ad)^ib^jc^k : i,j,k \geq
0\rangle$ is free on the generators $b^jc^k$, so $\im E_0$ is free on
the image of $b^jc^k$ for $j,k$ not both zero.  Using the description of
the differential above, we get:
\begin{lem}
  $\im E_0$ is freely generated by
  $b^jc^k(a\Omega^{10}_{00}-d\Omega^{00}_{01})$ for $j,k$ not both zero.
\end{lem} 
In the above Lemma and from now on, the terms `free
generators' and `freely generated'  refer to the $k[\Delta]$-module
structure.

\begin{lem}
  $\ker E_1$ is freely generated by
  $b\Omega^{01}_{00}+a\Omega^{10}_{00}$,
  $a\Omega^{10}_{00}+c\Omega^{00}_{10}$ and all
  $b^jc^k(a\Omega^{10}_{00}-d\Omega^{00}_{01})$ with $j,k \geq 0$.
\end{lem}

\begin{proof}
  $K_0(M)_1$ is spanned by elements of the form
  $(ad)^iab^jc^k\Omega^{10}_{00}$,
  $(ad)^ib^jc^k\Omega^{01}_{00}$, $(ad)^ib^jc^k\Omega^{00}_{10}$,
  $(ad)^idb^jc^k\Omega^{00}_{01}$ and is therefore freely generated by
  the elements $ab^jc^k\Omega^{10}_{00}$, $b^jc^k\Omega^{01}_{00}$,
  $b^jc^k\Omega^{00}_{10}$, $db^jc^k\Omega^{00}_{01}$. The images under
  $E_1$ of these free generators are
  \begin{align*}
E_1ab^jc^k\Omega^{10}_{00}=E_1db^jc^k\Omega^{00}_{01} 
  &= ((j+k+2)b^{j+1}c^{j+1} + (j+k)\Delta b^jc^k )
    \Omega^{10}_{01}\\ 
    E_1 b^jc^k\Omega^{01}_{00}& =
    (j+k-1)b^jc^k(a\Omega^{11}_{00}+d\Omega^{01}_{01})-2b^jc^{k+1}\Omega^{10}_{01} \\
    E_1 b^jc^k\Omega^{00}_{10} &=
    (j+k-1)b^jc^k(a\Omega^{10}_{10}+d\Omega^{00}_{11})-2b^{j+1}c^{k}\Omega^{10}_{01} 
  \end{align*}
  Therefore if $p_{jk},q_{jk},r_{jk},s_{jk}$ are polynomials, the image
  of \begin{equation}\label{would_be_kernel_element} \sum p_{jk}(\Delta)ab^jc^k \Omega^{10}_{00} +
    q_{jk}(\Delta)b^jc^k \Omega^{01}_{00}+r_{jk}(\Delta)b^jc^k
    \Omega^{00}_{10}+s_{jk}(\Delta)db^jc^k
    \Omega^{00}_{01}\end{equation}
    is
    \begin{multline*}
      \sum
      (j+k-1)q_{jk}(\Delta)(ab^jc^k\Omega^{11}_{00}+db^jc^k\Omega^{01}_{01})
      + \sum
      (j+k-1)r_{jk}(\Delta)(ab^jc^k\Omega^{10}_{10}+db^jc^k\Omega^{00}_{11})
      \\
      +\sum (p_{jk}(\Delta)+s_{jk}(\Delta) ) (
      (j+k+2)b^{j+1}c^{k+1} + (j+k)\Delta b^jc^k )\Omega^{10}_{01} \\ 
      - \sum (2q_{jk}(\Delta)b^{j}c^{k+1} + 2r_{jk}(\Delta)b^{j+1}c^k)
      \Omega^{10}_{01}
    \end{multline*}

    The terms $ab^jc^k\Omega^{10}_{10}$ are free generators of the
    $\Omega^{10}_{10}$ component of $K_0(M)_2$, and so the only way for
    the $\Omega^{10}_{10}$ component of this expression to vanish is if
    $r_{jk}=0$ for $j+k \neq 1$.  Similarly we must have $q_{jk}=0$ for
    $j+k \neq 1$.

    The coefficient of $c^2\Omega^{10}_{01}$ in this sum is
    $2q_{01}(\Delta)c^2\Omega^{10}_{01}$, therefore $q_{01}=0$, and
    similarly $r_{10}=0$.  The condition for
    (\ref{would_be_kernel_element}) to be in $\ker E_1$ is therefore
    that $r_{jk}=q_{jk}=q_{01}=r_{10}$ for $j+k \neq 1$ and
    \begin{align*}
      -q_{10}-r_{01} + p_{00}+ s_{00} &= 0 \\
      p_{jk}+s_{jk} &= 0
    \end{align*}
    for $j,k$ not both zero.  This is equivalent to our claim about the
    kernel generators.
\end{proof}

We say an element $\sum
\alpha_{ijkl}^{i'j'k'l'}f^{ij}_{kl}\Omega^{i'j'}_{k'l'}$ of $\ker E_*$ \textbf{lifts
trivially} if $\sum
\alpha_{ijkl}^{i'j'k'l'}\sff^{ij}_{kl}\tilde\Omega^{i'j'}_{k'l'}$ is a
cocycle for $\bar{K}_{\sM^e}(\sM)$.

\begin{cor}
  $\HP^1(M)$ is freely generated by the images of
  $b\Omega^{01}_{00}+a\Omega^{10}_{00}$,
  $a\Omega^{10}_{00}+c\Omega^{00}_{10}$ and
  $a\Omega^{10}_{00}-d\Omega^{00}_{01}$. Each of these elements lifts trivially
  to an element of $\HH^1(\sM)$.
\end{cor}

\begin{proof}
  The first part follows immediately from the previous two lemmas.  The
  second follows from the description of the differential in
  $\bar{K}_{\sM^e}(\sM)$ given earlier.
\end{proof}

\begin{lem}
  $\im E_1$ is freely generated by $c^2\Omega^{10}_{01}$,
  $b^2\Omega^{10}_{01}$, $(b^{j+1}c^{k+1} +\frac{j+k}{j+k+2}\Delta b^jc^k
  )\Omega^{10}_{01}$ for $j,k \geq 0$ and for $j+k\neq 1$,
  \begin{gather*}
    b^jc^k(a\Omega^{11}_{00}+d\Omega^{01}_{01})-\frac{2}{j+k-1}\Omega^{10}_{01}
    \\
    b^jc^k(a\Omega^{10}_{10}+d\Omega^{00}_{11})-\frac{2}{j+k-1}\Omega^{10}_{01}.
  \end{gather*} 
\end{lem}

\begin{proof}
From our description of $\ker E_1$ it follows that $\im E_1$ is freely
generated by the images of $ab^jc^k\Omega^{10}_{00}$ for any $j,k$,
$b^jc^k\Omega^{01}_{00}$ and $b^jc^k\Omega^{00}_{10}$ for $j+k \neq 1$
and $c\Omega^{01}_{00}$ and $b\Omega^{00}_{10}$.  These images are
non-zero scalar multiples of the given elements.
\end{proof}

\begin{lem}
  $\ker E_2$ is freely generated by
  $b^kc^k(a\Omega^{11}_{00}+d\Omega^{01}_{01})$,
  $b^jc^k(a\Omega^{10}_{10}+d\Omega^{00}_{11})$,
  $b^jc^k\Omega^{10}_{01}$ for $j,k \geq 0$ and
  $bc\Omega^{01}_{10}-ab\Omega^{11}_{00} + ac\Omega^{10}_{10}$.
\end{lem}

\begin{proof}
  $E_2$ kills $b^jc^k\Omega^{10}_{01}$ and acts on the other 
  free generators of $K_0(M)_2$ as follows:
  \begin{align*}
    E_2ab^jc^k\Omega^{11}_{00}=-E_2db^jc^k\Omega^{01}_{01} &=(
    (j+k+1)b^{k+1}c^{k+1} + (j+k-1)\Delta b^jc^k) \Omega^{11}_{01} \\
    E_2ab^jc^k\Omega^{10}_{10}=-E_2db^jc^k\Omega^{00}_{11} &=(
    (j+k+1)b^{k+1}c^{k+1} + (j+k-1)\Delta b^jc^k) \Omega^{10}_{11} \\
    E_2 b^jc^k\Omega^{01}_{10} &=
    (j+k-2)b^jc^k(a\Omega^{11}_{10}-d\Omega^{01}_{11})+2b^jc^k(c\Omega^{10}_{11}-b\Omega^{11}_{01})
  \end{align*}
As in the calculation of $\ker E_1$ we see by considering the
$\Omega^{11}_{10}$ term that a kernel element of the form
\[ \sum p_{jk}(\Delta)ab^jc^k\Omega^{11}_{00} + q_{jk}(\Delta)ab^jc^k\Omega^{10}_{10}
  +
  r_{jk}(\Delta)b^jc^k\Omega^{01}_{10}+s_{jk}(\Delta)db^jc^k\Omega^{01}_{01}+t_{jk}(\Delta)db^jc^k\Omega^{00}_{11}\]
  must have $r_{jk}=0$ unless $j+k=2$, and since the
  $b^3\Omega^{11}_{01}$ and $c^3\Omega^{10}_{11}$ coefficients must
  vanish $r_{jk}=0$ unless $j=k=1$.
  The condition for an element obeying these restrictions on
  $r_{jk}$ to lie in
the kernel is that 
\begin{alignat*}{2}
  r_{11}+q_{01}-t_{01}&=0 &\quad \quad -r_{11}+p_{10}-s_{10} &=0 \\
  p_{jk}-s_{jk} &=0 &  q_{lm}-r_{lm} &= 0
\end{alignat*}
for $(j,k)\neq (1,0)$ and $(l,m)\neq(0,1)$. It follows that the kernel
is as claimed.
\end{proof}

\begin{cor}
  $\HP^2(M)$ is freely generated by the images of  \begin{gather*} b(a\Omega^{11}_{00}+d\Omega^{01}_{01}), b(a
    \Omega^{10}_{10}+d\Omega^{00}_{11}), c(a\Omega^{11}_{00}+d\Omega^{01}_{01}), c(a
    \Omega^{10}_{10}+d\Omega^{00}_{11}),\\
    bc\Omega^{01}_{10}-ab\Omega^{11}_{00}+ac\Omega^{10}_{10},
    b^r\Omega^{10}_{01}, c^r\Omega^{10}_{01}\end{gather*}
  for $r\neq 2$. Each of these elements lifts trivially to an element of
  $\HH^2(\sM)$.
\end{cor}
\begin{proof}
  The first statement follows from our calculation of $\ker E_2$ and
  $\im E_1$, the second by computing using the description of the
  differential on $\bar{K}_{\sM^e}(\sM)$ given earlier.
\end{proof}

\begin{lem}
  $\im E_2$ is freely generated by $(b^{j+1}c^{k+1} +
  \frac{j+k-1}{j+k+1}\Delta b^jc^k )\Omega^{11}_{01}$,  $(b^{j+1}c^{k+1}+\frac{j+k-1}{j+k+1}\Delta
  b^jc^k)\Omega^{10}_{11}$ for any $j,k \geq 0$,
  $b^jc^k(a\Omega^{11}_{10}-d\Omega^{01}_{11})+2b^jc^k(c\Omega^{10}_{11}-b\Omega^{11}_{01})$
  for $j+k\neq 2$, 
  $b^3\Omega^{11}_{01}$ and $c^3\Omega^{10}_{11}$.
\end{lem}

\begin{proof}
  The computation of $\ker E_2$ shows that $\im E_2$ is free on the
  images of $b^jc^ka\Omega^{11}_{00}$, $b^jc^ka\Omega^{10}_{10}$,
  $b^jc^k\Omega^{01}_{10}$ for $j+k\neq 2$, 
  $b^2\Omega^{01}_{01}+ab\Omega^{10}_{10}$ and
  $c^2\Omega^{01}_{10} -ac\Omega^{11}_{00}$
  which are, up to a scalar multiple, the given generators.
\end{proof}

\begin{lem}
  $\ker E_3$ is freely generated by $b^jc^k\Omega^{10}_{11}, b^jc^k
  \Omega^{11}_{01}$, and
  $b^jc^k(a\Omega^{11}_{10}-d\Omega^{01}_{11})$ for all $j,k \geq 0$.
\end{lem}

\begin{proof}
  $E_3$ kills $b^jc^k\Omega^{10}_{11}$ and $b^jc^k\Omega^{11}_{01}$ and
  sends $b^jc^ka\Omega^{11}_{10}$ and $b^jc^kd\Omega^{01}_{11}$ to $((j+k)b^{j+1}c^{k+1} + (j+k-2)\Delta
  b^jc^k ) \Omega^{11}_{11}$.  The result follows immediately.
\end{proof}

Again we can read off the cohomology group:
\begin{cor}
  $\HP^3(M)$ is freely generated by the images of 
  \[ b^j\Omega^{11}_{01}, c^k\Omega^{11}_{01},
		b^k\Omega_{11}^{10}, c^j\Omega_{11}^{10} \]
	for $j \neq 3$ together with
	$ax\Omega^{11}_{10}-dx\Omega_{11}^{01}$ for $x = b^2,
	bc, c^2$. Each of these elements lifts trivially to an element
	of $\HH^3(\sM)$.
\end{cor}

\begin{lem}
  $\im E_3$ is freely generated by $\Delta\Omega^{11}_{11}$ and
  $(b^{j+1}c^{k+1}+\frac{j+k-2}{j+k}b^jc^k)\Omega^{11}_{11}$ for $j,k$
  not both zero.
\end{lem}

\begin{proof}
  Our computation of $\ker E_3$ shows that $\im E_3$ is free on the
  images of $b^jc^ka\Omega^{11}_{10}$, which up to a nonzero scalar
  multiple are the generators given.
\end{proof}

\begin{cor}
  $\HP^4(M)$ is generated as a $k[\Delta]$-module by $\Omega^{11}_{11}$,
  $b^j\Omega^{11}_{11}$ and $c^j\Omega^{11}_{11}$ for $j,k\geq 0$.
  $k[\Delta]$ acts trivially on $\Omega^{11}_{11}$ and freely on the
    other generators.
\end{cor}

Since $E_4=0$, all kernel elements lift trivially.

\section{Acknowledgements}
I would like to thank St\'ephane Launois for suggesting this problem and
for many helpful suggestions.   This research was financially supported
through his EPSRC First Grant EP/I018549/1.

\bibliography{scl}
\bibliographystyle{alpha}

\end{document}